\documentclass{article} 

 \usepackage[usenames,dvipsnames]{pstricks}
 \usepackage{epsfig}
 \usepackage{pst-grad} 
 \usepackage{pst-plot} 
 \usepackage[space]{grffile} 
 \usepackage{etoolbox} 
 \makeatletter 
 \patchcmd\Gread@eps{\@inputcheck#1 }{\@inputcheck"#1"\relax}{}{}
 \makeatother

\usepackage[usenames,dvipsnames]{pstricks}
\usepackage{epsfig}
\usepackage{pst-grad} 
\usepackage{pst-plot} 
\usepackage[space]{grffile} 
\usepackage{etoolbox} 
\makeatletter 
\patchcmd\Gread@eps{\@inputcheck#1 }{\@inputcheck"#1"\relax}{}{}
\makeatother

 \usepackage[usenames,dvipsnames]{pstricks}
 \usepackage{epsfig}
 \usepackage{pst-grad} 
 \usepackage{pst-plot} 
 \usepackage[space]{grffile} 
 \usepackage{etoolbox} 
 \makeatletter 
 \patchcmd\Gread@eps{\@inputcheck#1 }{\@inputcheck"#1"\relax}{}{}
 \makeatother

\usepackage{amsmath}
\usepackage{amssymb}
\usepackage{amsthm}
\usepackage[latin1]{inputenc}
     
\usepackage{graphicx}

\usepackage[usenames,dvipsnames]{pstricks}
\usepackage{epsfig}
\usepackage{pst-grad} 
\usepackage{pst-plot} 

\usepackage{ifthen}
\usepackage{color}

\newcommand{\intav}[1]{\mathchoice {\mathop{\vrule width 6pt height 3 pt depth  -2.5pt
\kern -8pt \intop}\nolimits_{\kern -6pt#1}} {\mathop{\vrule width
5pt height 3  pt depth -2.6pt \kern -6pt \intop}\nolimits_{#1}}
{\mathop{\vrule width 5pt height 3 pt depth -2.6pt \kern -6pt
\intop}\nolimits_{#1}} {\mathop{\vrule width 5pt height 3 pt depth
-2.6pt \kern -6pt \intop}\nolimits_{#1}}}

\def\polhk#1{\setbox0=\hbox{#1}{\ooalign{\hidewidth\lower1.5ex\hbox{`}\hidewidth\crcr\unhbox0}}}

\def\XXint#1#2#3{{\setbox0=\hbox{$#1{#2#3}{\int}$ }
\vcenter{\hbox{$#2#3$ }}\kern-.6\wd0}}

\renewcommand{\div}{\operatorname{div}}

\renewcommand{\div}{\operatorname{div}}

\newcommand{\dist}{\operatorname{dist}}

\newcommand{\Id}{\operatorname{Id}}

\newcommand{\vol}{\operatorname{vol}}

\usepackage{setspace}

\newtheorem{teo}{Theorem}
\newtheorem{example}{Example}
\newtheorem{Definition}{Definition}
\newtheorem{Lemma}{Lemma}
\newtheorem{Corollary}{Corollary}
\newtheorem{Proposition}{Proposition}
\newtheorem{Remark}{Remark}
\newtheorem{Assumption}{A}

\makeatletter
\patchcmd{\env@cases}{1.2}{.8}{}{}
\makeatother

\begin{document}

\title{A fully nonlinear free transmission problem}
\author{Edgard A. Pimentel and Makson S. Santos}

\date{\today} 

\maketitle

\begin{abstract}

\noindent We examine a free transmission problem driven by fully nonlinear elliptic operators. Since the transmission interface is determined endogeneously, our analysis is two-fold: we study the regularity of the solutions and some geometric properties of the  free boundary. By relating our problem with a pair of viscosity inequalities, we prove that strong solutions are of class $C ^{1,1}$, locally. As regards the free boundary, we start by establishing weak results, such as its non-degeneracy, and proceed with the characterization of global solutions. 
\medskip


\medskip 

\end{abstract}

\vspace{.1in}

\section{Introduction}\label{sec_introduction}

We consider a fully nonlinear transmission problem of the form
\begin{equation}\label{eq_main}
	F_1(D^2u)\chi_{\left\{u>0\right\}}+F_2(D^2u)\chi_{\left\{u<0\right\}}=1 \hspace{.3in} \mbox{in} \hspace{.3in} \Omega^-(u)\cup\Omega^+(u)
\end{equation}
where $F_1, F_2:\mathcal{S}(d)\to\mathbb{R}$ are $(\lambda,\Lambda)$-elliptic operators, $\Omega^-(u):=\{x\in B_1\mid u<0\}$, and $\Omega^+(u):=\{x\in B_1\mid u>0\}$. We prove optimal regularity results for the strong solutions to \eqref{eq_main} and examine the associated free boundary. In particular, we prove that solutions are locally of class $C ^{1,1}$ and establish non-degeneracy of the free interface. The latter unlocks the analysis of global solutions.

We emphasize the operators $F_1$ and $F_2$ are not supposed to satisfy any proximity regime, or even to be comparable. As a consequence, \eqref{eq_main} differs from the usual obstacle problem. In fact, discontinuities arise in the \emph{diffusion process}, as the solutions change sign. 

Transmission problems comprise a class of models aimed at examining a variety of phenomena in heterogeneous media. The problems under the scope of this formulation include thermal and electromagnetic conductivity, composite materials and, more generally, diffusion processes driven by discontinuous laws.

Given a domain $\Omega\subset\mathbb{R}^d$, it gets split into mutually disjoint subregions $\Omega_i\Subset\Omega$ for $i=1, \ldots, k$, for some $k\in\mathbb{N}$. The mechanism governing the problem is smooth within $\Omega_i$, though possibly discontinuous across $\partial\Omega_i$. A paramount, subtle, aspect of the theory concerns the nature of those subregions.

In fact, $\left(\Omega_i\right)_{i=1}^k$ and the geometry of $\partial\Omega_i$ can be prescribed a priori. The alternative is $\left(\Omega_i\right)_{i=1}^k$ to be determined endogenously. The latter setting frames the theory in the context of free boundary problems. Both cases differ substantially; as a consequence, their analysis also requires distinct techniques. The vast majority of former studies on transmission problems presupposes \emph{a priori knowledge} of the subregions $\Omega_i$ and their geometric properties. A work-horse of the theory is the divergence-form equation
\begin{equation}\label{eq_tm}
	\div\left(a(x)Du\right) = 0 \hspace{.2in} \mbox{in} \hspace{.2in} \Omega,
\end{equation}
where the matrix-valued function $a(\cdot)$ is defined as
\[
	a(x) := a_i \hspace{.1in} \mbox{for} \hspace{.1in} x \in \Omega_i,
\]
for constant matrices $a_i$ and $i=1, \ldots, k$. Though smooth within every $\Omega_i$, the coefficients of \eqref{eq_tm} can be discontinuous across $\partial\Omega_i$. This feature introduces genuine difficulties in the analysis. 

The first formulation of a transmission problem appeared in \cite{Picone1954} and addressed a topic in the realm of material sciences. More precisely, in elasticity theory. In that paper, the author proves the uniqueness of solutions for a model consisting of two subregions, which are known a priori. The  existence of solutions is discussed in \cite{Picone1954}, although not examined in detail. See also \cite{Picone1936}.

The formulation in \cite{Picone1954} motivated a number of subsequent studies \cite{Borsuk1968,Campanato1957,Campanato1959,Campanato1959a,Lions1955,Iliin-Shismarev1961,Oleinik1961,Schechter1960,Sheftel1963,Stampacchia1956}. Those papers present a wide range of developments, including the existence of solutions for the transmission problem in \cite{Picone1954} and the analysis of several variants. We refer the reader to \cite{Borsuk2010} for an account of those results and methods.

Estimates and regularity results for the solutions to transmission problems have also been treated in the literature. In \cite{Li-Vogelius2000} the authors consider a bounded subdomain $\Omega\subset\mathbb{R}^d$, which is split into a finite number of subregions $\Omega_1, \Omega_2, \ldots, \Omega_k$, known a priori. The motivation is in the study of composite materials with closely spaced inclusions. A two-dimensional example is the cross-section of a fiber-reinforced material; see Figure \ref{fig_frm}. The mathematical analysis amounts to the study of
\begin{equation}\label{eq_livogelius}
	\frac{\partial}{\partial x_i}\left(a(x)\frac{\partial}{\partial x_j}u\right) = f \hspace{.25in} \mbox{in} \hspace{.25in} \Omega,
\end{equation}
where 
\begin{equation*}
	a(x) := 
		\begin{cases}
			a_i(x)& \hspace{.3in}\mbox{for}\hspace{.3in}x\in \Omega_i,\;\;i=1, \ldots, k\\
			a_{k+1}(x)& \hspace{.3in}\mbox{for}\hspace{.3in}x\in \Omega\setminus\cup_{i=1}^k\Omega_i.
		\end{cases}
\end{equation*}

Under natural assumptions on the data, the authors establish local H\"older continuity for the gradient of the solutions. From the applied perspective, the gradient encodes information on the stress of the material. Their findings imply bounds on the gradient \emph{independent of the location of the fibers}. C.f. \cite{Bonnetier2000}.

\begin{figure}[h]
\centering

\psscalebox{.60 .60} 
{
\begin{pspicture}(0,-4.457367)(18.671059,4.457367)
\definecolor{colour0}{rgb}{0.8,0.8,0.8}
\definecolor{colour1}{rgb}{0.6,0.6,0.6}
\psframe[linecolor=black, linewidth=0.04, dimen=outer](8.929585,3.5630813)(4.272811,-1.093693)
\pscustom[linecolor=colour1, linewidth=0.1]
{
\newpath
\moveto(4.945456,0.82075864)
\lineto(4.900182,0.7108069)
\curveto(4.877545,0.6558313)(4.8387384,0.52971005)(4.822569,0.45856506)
\curveto(4.8063993,0.38742006)(4.774061,0.23219422)(4.7578917,0.1481134)
\curveto(4.7417226,0.06403259)(4.70615,-0.06855591)(4.6867466,-0.11706421)
\curveto(4.667343,-0.16557221)(4.596198,-0.26582214)(4.544456,-0.3175641)
\curveto(4.4927144,-0.36930603)(4.392464,-0.47602355)(4.3439565,-0.5309991)
\curveto(4.295448,-0.58597505)(4.2275367,-0.67975706)(4.2081337,-0.71856385)
\curveto(4.1887302,-0.7573706)(4.14669,-0.87378967)(4.1240525,-0.9514026)
\curveto(4.1014156,-1.0290155)(4.043206,-1.1454346)(4.0076337,-1.1842413)
\curveto(3.9720612,-1.2230481)(3.891214,-1.2941931)(3.8459399,-1.326532)
\curveto(3.8006659,-1.3588709)(3.677779,-1.41708)(3.600166,-1.4429511)
\curveto(3.5225532,-1.468822)(3.3834968,-1.5334997)(3.3220532,-1.5723059)
\curveto(3.2606096,-1.6111121)(3.176529,-1.6854913)(3.1538916,-1.7210639)
\curveto(3.1312547,-1.7566364)(3.0762787,-1.8374829)(3.04394,-1.8827575)
\curveto(3.0116014,-1.9280316)(2.9275208,-2.012112)(2.875779,-2.0509186)
\curveto(2.8240368,-2.0897253)(2.7334886,-2.1479347)(2.6946821,-2.1673377)
\curveto(2.6558754,-2.1867406)(2.587964,-2.225548)(2.5588593,-2.2449512)
\curveto(2.5297546,-2.2643542)(2.4359722,-2.2999268)(2.371295,-2.3160963)
\curveto(2.3066175,-2.3322656)(2.196666,-2.3516686)(2.1513917,-2.3549023)
}
\pscustom[linecolor=colour1, linewidth=0.1]
{
\newpath
\moveto(4.7643595,3.1620812)
\lineto(5.0295367,2.916307)
\curveto(5.162125,2.7934198)(5.44994,2.6284924)(5.605166,2.5864522)
\curveto(5.7603917,2.544412)(6.022335,2.4538636)(6.1290526,2.4053555)
\curveto(6.23577,2.3568473)(6.3780613,2.2695327)(6.413634,2.2307262)
\curveto(6.4492064,2.1919198)(6.465376,2.1046054)(6.445973,2.0560975)
\curveto(6.42657,2.0075893)(6.3554244,1.9008716)(6.3036823,1.8426617)
\curveto(6.2519403,1.784452)(6.1290526,1.6874359)(6.0579076,1.6486298)
\curveto(5.9867625,1.6098233)(5.844473,1.5322104)(5.773328,1.4934039)
\curveto(5.702183,1.4545975)(5.585763,1.3834522)(5.5404882,1.3511136)
\curveto(5.4952145,1.318775)(5.4240694,1.2637992)(5.398198,1.2411621)
\curveto(5.3723273,1.2185248)(5.330287,1.1603149)(5.3141174,1.1247424)
\curveto(5.297948,1.08917)(5.2365046,0.95011353)(5.1912303,0.8466296)
\curveto(5.1459565,0.74314576)(5.078045,0.56851685)(5.0554075,0.49737152)
\curveto(5.0327706,0.4262262)(5.0069,0.16751678)(5.003666,-0.020047607)
\curveto(5.000432,-0.2076123)(4.974561,-0.4436847)(4.951924,-0.49219298)
\curveto(4.929287,-0.540701)(4.8840127,-0.6215479)(4.861376,-0.65388644)
\curveto(4.8387384,-0.686225)(4.793464,-0.783241)(4.770827,-0.8479187)
\curveto(4.74819,-0.91259646)(4.7093835,-1.0225476)(4.693214,-1.0678222)
\curveto(4.677045,-1.1130964)(4.599432,-1.1874756)(4.537988,-1.2165802)
\curveto(4.476545,-1.2456849)(4.3730607,-1.2909595)(4.3310204,-1.3071289)
\curveto(4.28898,-1.3232983)(4.224303,-1.355636)(4.201666,-1.3718054)
\curveto(4.1790285,-1.3879749)(4.140222,-1.4623542)(4.1240525,-1.520564)
\curveto(4.1078835,-1.5787731)(4.082012,-1.6790234)(4.052908,-1.8051445)
}
\pscustom[linecolor=colour1, linewidth=0.1]
{
\newpath
\moveto(4.9583917,2.1013715)
\lineto(4.64794,1.9526135)
\curveto(4.4927144,1.8782343)(4.2566414,1.6162909)(4.1757946,1.4287262)
\curveto(4.0949483,1.2411618)(3.9979322,0.90160555)(3.9817626,0.74961364)
\curveto(3.9655933,0.59762144)(3.7844965,0.3227423)(3.619569,0.19985534)
\curveto(3.4546418,0.07696839)(3.176529,-0.0750238)(3.0633435,-0.10412842)
\curveto(2.950158,-0.13323334)(2.7528918,-0.15587036)(2.5006497,-0.13646716)
}
\pscustom[linecolor=colour1, linewidth=0.1]
{
\newpath
\moveto(2.3195531,-0.52453154)
\lineto(2.5588596,-0.56333804)
\curveto(2.6785126,-0.5827411)(3.0956821,-0.60861206)(3.3931983,-0.61507994)
\curveto(3.6907141,-0.6215479)(4.020569,-0.6021445)(4.052908,-0.5762735)
\curveto(4.0852466,-0.5504025)(4.172561,-0.45662048)(4.2275367,-0.3887091)
\curveto(4.2825127,-0.32079774)(4.3665934,-0.20114471)(4.395698,-0.14940277)
\curveto(4.424803,-0.097660825)(4.470077,0.005823059)(4.486246,0.057565004)
\curveto(4.5024157,0.10930695)(4.4927144,0.29040375)(4.466843,0.41975862)
\curveto(4.4409723,0.54911345)(4.392464,0.72697634)(4.3698273,0.7754846)
\curveto(4.34719,0.8239926)(4.337488,0.9274765)(4.350424,0.9824524)
\curveto(4.3633595,1.037428)(4.476545,1.1247424)(4.576795,1.1570812)
\curveto(4.677045,1.1894202)(4.9163513,1.2476296)(5.0554075,1.2735007)
\curveto(5.194464,1.2993716)(5.5825286,1.3252426)(5.831537,1.3252426)
\curveto(6.080545,1.3252426)(6.394231,1.3414117)(6.458908,1.3575811)
\curveto(6.5235853,1.3737506)(6.698214,1.4060894)(6.808166,1.4222589)
\curveto(6.9181175,1.4384284)(7.066875,1.4707665)(7.105682,1.4869361)
}
\pscustom[linecolor=colour1, linewidth=0.1]
{
\newpath
\moveto(7.558424,3.1103394)
\lineto(7.4096656,2.9486458)
\curveto(7.335287,2.867799)(7.2221007,2.6964037)(7.183295,2.6058555)
\curveto(7.1444883,2.515307)(7.0086656,2.136944)(6.9116497,1.8491297)
\curveto(6.814634,1.5613153)(6.627069,1.1926538)(6.536521,1.1118072)
\curveto(6.445973,1.0309604)(6.2745776,0.88867)(6.1937304,0.8272266)
\curveto(6.112883,0.7657831)(5.9964643,0.65259737)(5.9608917,0.6008554)
\curveto(5.925319,0.54911345)(5.8509398,0.4262262)(5.812134,0.35508117)
\curveto(5.773328,0.28393614)(5.6892467,0.1901538)(5.6439724,0.16751648)
\curveto(5.598698,0.14487946)(5.514618,0.08666992)(5.475811,0.05109741)
\curveto(5.4370046,0.015524902)(5.3432226,-0.08472534)(5.2882466,-0.14940277)
\curveto(5.23327,-0.2140802)(5.142722,-0.3531366)(5.1071496,-0.42751557)
\curveto(5.071577,-0.50189453)(5.0133677,-0.65712035)(4.990731,-0.7379669)
\curveto(4.9680934,-0.8188135)(4.929287,-0.9772736)(4.913118,-1.0548865)
\curveto(4.8969483,-1.1324993)(4.861376,-1.2456849)(4.8419724,-1.2812573)
\curveto(4.822569,-1.3168298)(4.774061,-1.3718054)(4.744956,-1.3912085)
}
\pscustom[linecolor=colour1, linewidth=0.1]
{
\newpath
\moveto(8.528585,2.2954037)
\lineto(8.095246,2.2177908)
\curveto(7.878577,2.1789844)(7.616634,2.094904)(7.5713596,2.0496297)
\curveto(7.5260854,2.0043554)(7.442004,1.8847022)(7.4031982,1.8103232)
\curveto(7.3643923,1.7359443)(7.302948,1.6033558)(7.280311,1.5451459)
\curveto(7.2576737,1.4869361)(7.2188673,1.344646)(7.2026978,1.2605652)
\curveto(7.186528,1.1764843)(7.1444883,1.0341944)(7.1186175,0.9759845)
\curveto(7.0927467,0.9177747)(7.018368,0.7884198)(6.9698596,0.7172748)
\curveto(6.9213514,0.6461298)(6.8372707,0.52971005)(6.801698,0.48443604)
\curveto(6.7661257,0.439162)(6.675577,0.34214568)(6.6206017,0.29040375)
\curveto(6.565626,0.23866181)(6.46861,0.15781493)(6.426569,0.12871033)
\curveto(6.3845286,0.09960541)(6.300448,0.05109741)(6.258408,0.031694032)
\curveto(6.2163677,0.012290649)(6.1225853,-0.023281861)(6.0708437,-0.039451294)
\curveto(6.0191016,-0.055620726)(5.922086,-0.08472534)(5.7862625,-0.12353168)
}
\psline[linecolor=black, linewidth=0.04](4.272811,3.5278027)(1.6857142,0.16457683)
\psline[linecolor=black, linewidth=0.04](8.917826,3.5278027)(6.3189692,0.17633636)
\psline[linecolor=black, linewidth=0.04](8.894307,-1.0819334)(6.330729,-4.4451594)
\psline[linecolor=black, linewidth=0.04](4.2845707,-1.093693)(1.6974738,-4.4451594)

\pscustom[linecolor=colour1, linewidth=0.1]
{
\newpath
\moveto(8.412167,-0.56451416)
\lineto(8.288691,-0.6159137)
\curveto(8.2269535,-0.6416135)(8.1505165,-0.71871275)(8.135817,-0.7701123)
\curveto(8.121117,-0.82151186)(8.091719,-0.9178857)(8.077019,-0.9628601)
\curveto(8.06232,-1.0078344)(8.024101,-1.0785096)(8.000583,-1.104209)
\curveto(7.977063,-1.1299084)(7.927085,-1.1780957)(7.900626,-1.2005835)
\curveto(7.8741674,-1.2230706)(7.821249,-1.2648327)(7.7947903,-1.2841077)
\curveto(7.7683315,-1.3033825)(7.7036543,-1.3387195)(7.665436,-1.3547815)
\curveto(7.6272173,-1.3708435)(7.4772835,-1.4029688)(7.3655677,-1.4190308)
\curveto(7.2538524,-1.4350928)(7.0656996,-1.4704303)(6.9892626,-1.4897052)
\curveto(6.912826,-1.5089802)(6.79229,-1.5571667)(6.7481923,-1.5860791)
\curveto(6.7040944,-1.6149914)(6.6276574,-1.6792407)(6.595319,-1.7145783)
\curveto(6.5629797,-1.7499157)(6.5041823,-1.8173773)(6.4777236,-1.849502)
\curveto(6.451265,-1.8816266)(6.4042263,-1.952301)(6.383647,-1.9908508)
\curveto(6.3630676,-2.0294006)(6.321909,-2.125774)(6.30133,-2.1835985)
\curveto(6.2807508,-2.2414234)(6.2454724,-2.3442223)(6.2307725,-2.3891969)
\curveto(6.216073,-2.4341712)(6.177855,-2.5369701)(6.154336,-2.594795)
\curveto(6.1308165,-2.6526196)(6.0837793,-2.7522058)(6.06026,-2.793968)
\curveto(6.036741,-2.83573)(5.983823,-2.919254)(5.954424,-2.9610162)
\curveto(5.925025,-3.0027783)(5.8662276,-3.0670276)(5.8368287,-3.0895154)
\curveto(5.80743,-3.1120026)(5.7456923,-3.1698267)(5.713353,-3.2051642)
\curveto(5.6810145,-3.2405016)(5.613397,-3.3176007)(5.578119,-3.3593628)
\curveto(5.5428405,-3.401125)(5.489922,-3.4782238)(5.472283,-3.5135608)
\curveto(5.4546437,-3.5488977)(5.4193654,-3.6292102)(5.4017262,-3.6741846)
\curveto(5.384087,-3.719159)(5.3546877,-3.802683)(5.3194094,-3.9183323)
}
\pscustom[linecolor=colour1, linewidth=0.1]
{
\newpath
\moveto(7.05982,-0.6350714)
\lineto(6.9657435,-0.7348755)
\curveto(6.9187055,-0.7847778)(6.833449,-0.8596301)(6.795231,-0.8845813)
\curveto(6.757012,-0.9095325)(6.6247168,-0.9500781)(6.530641,-0.9656726)
\curveto(6.436565,-0.9812671)(6.31015,-1.0249305)(6.277811,-1.0530005)
\curveto(6.2454724,-1.0810705)(6.1749153,-1.1808753)(6.136697,-1.2526093)
\curveto(6.0984783,-1.3243432)(6.036741,-1.4428613)(6.0132213,-1.4896442)
\curveto(5.989702,-1.536427)(5.9456043,-1.6175183)(5.925025,-1.6518261)
\curveto(5.9044456,-1.686134)(5.8603477,-1.7453918)(5.8368287,-1.7703431)
\curveto(5.813309,-1.7952942)(5.7486324,-1.8358399)(5.7074738,-1.8514344)
\curveto(5.666315,-1.8670288)(5.5604796,-1.9075738)(5.495802,-1.932525)
\curveto(5.431124,-1.9574761)(5.331168,-2.004259)(5.29589,-2.0260913)
\curveto(5.2606115,-2.0479236)(5.204754,-2.097826)(5.184175,-2.125896)
\curveto(5.1635957,-2.153966)(5.116557,-2.2537696)(5.0900984,-2.325504)
\curveto(5.0636396,-2.3972383)(5.004842,-2.5126367)(4.972503,-2.5563014)
\curveto(4.940164,-2.5999658)(4.875487,-2.6841748)(4.843148,-2.7247205)
\curveto(4.810809,-2.7652662)(4.7578917,-2.8401184)(4.737313,-2.8744264)
\curveto(4.7167335,-2.908734)(4.678515,-2.971112)(4.6608753,-2.9991822)
\curveto(4.643236,-3.0272522)(4.5873785,-3.1426501)(4.54916,-3.2299793)
\curveto(4.510942,-3.3173077)(4.449204,-3.4857268)(4.425685,-3.5668175)
\curveto(4.402166,-3.647909)(4.3639474,-3.7944958)(4.349248,-3.8599927)
\curveto(4.3345485,-3.925489)(4.2904506,-3.9909852)(4.202254,-3.9909854)
}
\pscustom[linecolor=colour1, linewidth=0.1]
{
\newpath
\moveto(5.9309053,-0.58803314)
\lineto(5.718153,-0.6334619)
\curveto(5.611777,-0.65617615)(5.4708996,-0.71012205)(5.4363995,-0.74135435)
\curveto(5.4018993,-0.77258664)(5.338649,-0.8520862)(5.309898,-0.900354)
\curveto(5.2811475,-0.9486218)(5.2380223,-1.0593543)(5.2236476,-1.1218182)
\curveto(5.2092724,-1.1842822)(5.1632714,-1.3092108)(5.131646,-1.3716748)
\curveto(5.100021,-1.4341394)(5.0310197,-1.5448712)(4.993644,-1.593139)
\curveto(4.956269,-1.6414069)(4.8987684,-1.7067102)(4.878643,-1.7237463)
\curveto(4.858518,-1.7407819)(4.795267,-1.7805316)(4.752142,-1.8032459)
\curveto(4.7090163,-1.82596)(4.6256404,-1.8742279)(4.58539,-1.8997815)
\curveto(4.5451393,-1.925335)(4.4847636,-1.9736029)(4.464638,-1.9963171)
\curveto(4.444513,-2.0190313)(4.407138,-2.0758166)(4.389888,-2.1098883)
\curveto(4.3726373,-2.14396)(4.338137,-2.2177808)(4.3208866,-2.2575305)
\curveto(4.3036366,-2.2972803)(4.266261,-2.362584)(4.2461357,-2.3881378)
\curveto(4.226011,-2.4136915)(4.17426,-2.4534411)(4.142635,-2.467638)
\curveto(4.1110096,-2.4818347)(4.059259,-2.510227)(4.0391335,-2.5244238)
\curveto(4.019008,-2.53862)(3.984508,-2.567013)(3.9701328,-2.5812092)
\curveto(3.9557574,-2.5954053)(3.8695068,-2.6237981)(3.797631,-2.6379943)
\curveto(3.725755,-2.6521912)(3.628004,-2.6805835)(3.602129,-2.6947803)
\curveto(3.5762537,-2.708977)(3.533128,-2.7515662)(3.5158777,-2.7799585)
\curveto(3.4986274,-2.8083515)(3.4641273,-2.8651373)(3.4468768,-2.8935304)
\curveto(3.4296267,-2.9219227)(3.3893766,-2.9985833)(3.3663764,-3.0468512)
\curveto(3.343376,-3.095119)(3.2916253,-3.1973329)(3.2628748,-3.2512794)
\curveto(3.2341244,-3.3052258)(3.1766238,-3.3989222)(3.1478736,-3.4386718)
\curveto(3.1191232,-3.4784217)(3.070248,-3.5494044)(3.0501227,-3.580636)
\curveto(3.0299976,-3.6118684)(2.9811218,-3.702726)(2.9523711,-3.7623506)
\curveto(2.923621,-3.821975)(2.8776205,-3.9099932)(2.8603704,-3.9383862)
\curveto(2.84312,-3.9667785)(2.8201199,-3.9951715)(2.8028698,-3.9951715)
}
\pscustom[linecolor=colour1, linewidth=0.1]
{
\newpath
\moveto(8.39923,0.63966185)
\lineto(8.192263,0.581452)
\curveto(8.0887785,0.5523474)(7.959424,0.4812024)(7.9335527,0.439162)
\curveto(7.907682,0.39712158)(7.859174,0.30980712)(7.836537,0.26453277)
\curveto(7.8139,0.21925843)(7.752456,0.115774535)(7.7136497,0.057565004)
\curveto(7.674844,0.0)(7.600464,-0.09119324)(7.5648913,-0.123531796)
\curveto(7.529319,-0.15587036)(7.380561,-0.17850769)(7.2673755,-0.16880615)
\curveto(7.15419,-0.15910462)(6.960158,-0.18174133)(6.8793116,-0.2140802)
\curveto(6.7984643,-0.24641907)(6.6723437,-0.3660721)(6.627069,-0.45338655)
\curveto(6.581795,-0.540701)(6.539755,-0.7735394)(6.5429883,-0.91906375)
\curveto(6.5462217,-1.0645881)(6.539755,-1.2521527)(6.530053,-1.2941931)
\curveto(6.5203514,-1.3362329)(6.4492064,-1.4138465)(6.3877625,-1.4494189)
\curveto(6.3263187,-1.4849914)(6.1743274,-1.5529028)(6.083779,-1.5852417)
\curveto(5.99323,-1.6175805)(5.7442226,-1.6660876)(5.585763,-1.682257)
\curveto(5.4273033,-1.6984265)(5.239739,-1.733999)(5.2106338,-1.7534021)
\curveto(5.181528,-1.7728052)(5.1330204,-1.8245472)(5.1136174,-1.856886)
\curveto(5.0942144,-1.8892249)(5.052174,-1.9636041)(5.0295367,-2.0056446)
\curveto(5.0069,-2.047685)(4.9648595,-2.1414673)(4.945456,-2.1932092)
\curveto(4.926053,-2.2449512)(4.9034157,-2.5651045)(4.900182,-2.8335156)
\curveto(4.8969483,-3.1019268)(4.8969483,-3.5223303)(4.9066496,-3.9783058)
}
\pscustom[linecolor=colour1, linewidth=0.1]
{
\newpath
\moveto(7.770226,-1.5777069)
\lineto(7.589129,-1.6617872)
\curveto(7.4985805,-1.7038275)(7.3659916,-1.8008435)(7.3239512,-1.8558197)
\curveto(7.281912,-1.9107959)(7.207533,-2.0433843)(7.1751943,-2.1209974)
\curveto(7.142855,-2.1986103)(7.0135,-2.4088116)(6.9164844,-2.5414002)
\curveto(6.8194675,-2.6739893)(6.622202,-2.8971264)(6.5219517,-2.9876745)
\curveto(6.4217014,-3.0782228)(6.2470727,-3.25932)(6.1726933,-3.349868)
\curveto(6.0983143,-3.4404163)(6.0045323,-3.6118116)(5.9851294,-3.6926587)
\curveto(5.9657264,-3.7735052)(5.923685,-3.8705213)(5.8557744,-3.9190295)
}
\psbezier[linecolor=black, linewidth=0.1, arrowsize=0.05291667cm 2.0,arrowlength=1.4,arrowinset=0.0]{<-}(13.314132,0.07095534)(11.267385,0.050926894)(11.287692,0.0780358)(9.414439,0.05806425652360076)
\psframe[linecolor=black, linewidth=0.04, dimen=outer](18.67106,1.871284)(14.014286,-2.7854903)
\psellipse[linecolor=colour1, linewidth=0.04, fillstyle=solid,fillcolor=colour1, dimen=outer](14.904285,1.1045098)(0.21,0.21)
\psellipse[linecolor=colour1, linewidth=0.04, fillstyle=solid,fillcolor=colour1, dimen=outer](15.644286,0.22450979)(0.15,0.21)
\psellipse[linecolor=colour1, linewidth=0.04, fillstyle=solid,fillcolor=colour1, dimen=outer](16.754286,0.8245098)(0.22,0.13)
\psellipse[linecolor=colour1, linewidth=0.04, fillstyle=solid,fillcolor=colour1, dimen=outer](16.884285,0.084509805)(0.15,0.17)
\psellipse[linecolor=colour1, linewidth=0.04, fillstyle=solid,fillcolor=colour1, dimen=outer](17.164286,-0.7854902)(0.13,0.18)
\psellipse[linecolor=colour1, linewidth=0.04, fillstyle=solid,fillcolor=colour1, dimen=outer](18.024286,-2.3954902)(0.19,0.13)
\psellipse[linecolor=colour1, linewidth=0.04, fillstyle=solid,fillcolor=colour1, dimen=outer](17.284286,-2.4254904)(0.13,0.2)
\psellipse[linecolor=colour1, linewidth=0.04, fillstyle=solid,fillcolor=colour1, dimen=outer](16.514286,-2.4354904)(0.22,0.09)
\psellipse[linecolor=colour1, linewidth=0.04, fillstyle=solid,fillcolor=colour1, dimen=outer](15.444286,-2.2554903)(0.23,0.17)
\psellipse[linecolor=colour1, linewidth=0.04, fillstyle=solid,fillcolor=colour1, dimen=outer](14.764286,-2.3454902)(0.17,0.18)
\psellipse[linecolor=colour1, linewidth=0.04, fillstyle=solid,fillcolor=colour1, dimen=outer](15.024286,-1.2454902)(0.21,0.22)
\psellipse[linecolor=colour1, linewidth=0.04, fillstyle=solid,fillcolor=colour1, dimen=outer](16.204287,-1.2054902)(0.17,0.18)
\psellipse[linecolor=colour1, linewidth=0.04, fillstyle=solid,fillcolor=colour1, dimen=outer](18.084286,0.7445098)(0.17,0.17)
\psellipse[linecolor=colour1, linewidth=0.04, fillstyle=solid,fillcolor=colour1, dimen=outer](17.434286,1.3145097)(0.22,0.2)
\psframe[linecolor=black, linewidth=0.04, dimen=outer](6.3424883,0.19985542)(1.6857142,-4.4569187)
\end{pspicture}
}

\caption{The cross-section of a fiber-reinforced material provides an example in $\mathbb{R}^2$ of a bounded domain with a finite number of inclusions. The grey subregions in the cross-section represent the fibers, whereas the remainder of the material is the matrix.}\label{fig_frm}
\end{figure}
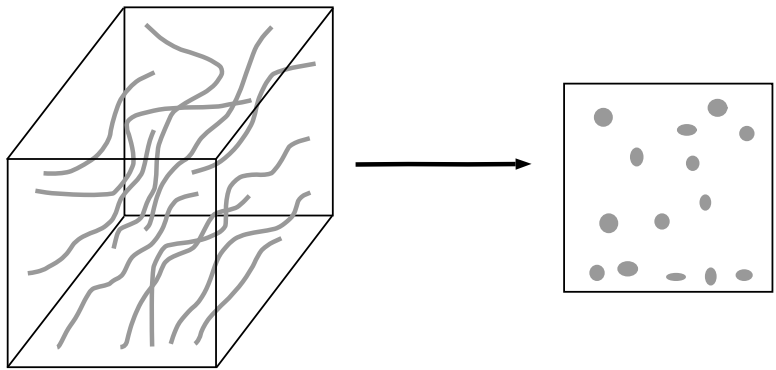

The vectorial setting is the subject of \cite{Li-Nirenberg2003}. In that paper the authors extend the developments reported in \cite{Li-Vogelius2000} to systems. Moreover, they produce bounds for higher derivatives of the solutions.

In \cite{Bao-Li-Yin2009} the authors consider a domain with two subregions, which are supposed to be $\varepsilon$-apart, for some $\varepsilon>0$. Within each subregion, the divergence-form equation is governed by a constant coefficient $k$. Conversely, outside those subregions the diffusivity coefficient is equal to $1$. By setting $k=+\infty$, the authors frame the problem in the context of perfect conductivity. 

In this setting, it is known that bounds on the gradient deteriorate as the two subregions approach each other. The analysis in \cite{Bao-Li-Yin2009} yields blow up rates for the gradient bounds as $\varepsilon\to 0$. The case of multiple inclusions, covering perfect conductivity and insulation $(k=0)$, is discussed in \cite{Bao-Li-Yin2010}. See also \cite{Briane-Capdeboscq-Nguyen2013}.

Recently, new developments have been obtained under minimal regularity requirements for the transmission interfaces. In \cite{Caffarelli-Carro-Stinga2020} the authors consider a smooth and bounded domain $\Omega$ and fix $\Omega_1\Subset\Omega$, defining $\Omega_2:=\Omega\setminus\overline{\Omega}_1$. They suppose the boundary of the transmission interface $\partial\Omega_1$ to be of class $C ^{1,\alpha}$ and prove existence, uniqueness and $C ^{1,\alpha}(\overline{\Omega}_i)$-regularity of the solutions to the problem, for $i=1, 2$. Their argument imports regularity from flat problems, through a new stability result; see \cite[Theorem 4.2]{Caffarelli-Carro-Stinga2020}. 

Another class of transmission problems concerns models where the subregions of interest are determined endogenously. For example, given $\Omega\subset\mathbb{R}^d$, one would consider
\[
	\Omega_1 : \left\{x \in \Omega \mid u(x) < 0\right\} \hspace{.3in} \mbox{and} \hspace{.3in} \Omega_2 : \left\{x \in \Omega \mid u(x) > 0\right\},
\]
where $u:\Omega\to\mathbb{R}$ solves a prescribed equation. Roughly speaking, knowledge of the solution is required to determine the subregions of the domain where distinct diffusion phenomena take place. In this context, a further structure arises, namely, the free interface, or free boundary. Here, in addition to the analysis of the solutions, properties of the free boundary are also of central interest. 

In \cite{Amaral-Teixeira2015} the authors examine a transmission problem with free interface. They consider the functional
\begin{equation}\label{eq_func}
	I(v) := \int_{\Omega}\frac{1}{2}\left\langle A(x,v)Dv, Dv\right\rangle + \Lambda(v) + fv dx,
\end{equation}
where
\[
	A(x,u) := A_+(x)\chi_{\left\lbrace u>0\right\rbrace} + A_-(x)\chi_{\left\lbrace u\leq0\right\rbrace},
\]
\[
	\Lambda(u) := \lambda_+(x)\chi_{\left\lbrace u>0\right\rbrace} + \lambda_-(x)\chi_{\left\lbrace u\leq0\right\rbrace},
\]
and
\[
f := f_+(x)\chi_{\left\lbrace u>0\right\rbrace} + f_-(x)\chi_{\left\lbrace u\leq0\right\rbrace}
\]
with $A_\pm$ matrix-valued mappings and $\lambda_\pm$ and $f_\pm$ given functions. Local minimizers for \eqref{eq_func} satisfy 
\[
	\begin{split}
		\div\left(A_+(x)Du\right) = f_+& \hspace{.3in} \mbox{in} \hspace{.3in} \Omega_+ := \left\{u>0\right\},\\
		\div\left(A_-(x)Du\right) = f_-& \hspace{.3in} \mbox{in} \hspace{.3in} \Omega_- := \left\{u<0\right\}^\circ,
	\end{split}
\]
while Hadamard's-type of arguments yield a flux condition across the free interface $F(u):=\partial \Omega_+\cap\Omega$, depending on $\lambda_+$ and $\lambda_-$. The authors prove the existence of minimizers, with $L^\infty$-bounds. In fact the proof of existence bypasses the lack of convexity of the functional and yields estimates in $L^\infty$ as a by-product. Those local minima are proved to have a local modulus of continuity. Under the assumption that $A_+$ and $A_-$ are close, in a sense made precise in that paper, the authors prove that solutions are indeed asymptotically Lipschitz. 

The problem examined in \cite{Amaral-Teixeira2015} profits from the existence of an associated functional and the properties derived for its minima. We remark those structures are not available in the context of \eqref{eq_main}.

In the present paper we study $W^{2,d}$-strong solutions to \eqref{eq_main}. We start by noticing that a $W^{2,d}$-solution to \eqref{eq_main} is a continuous viscosity solution to
\[
	\min\left(F_1(D^2u),F_2(D^2u)\right)\leq 1\hspace{.4in}\mbox{in}\hspace{.1in}B_1
\]
and
\[
	\max\left(F_1(D^2u),F_2(D^2u)\right)\geq -1\hspace{.4in}\mbox{in}\hspace{.1in}B_1.
\]

In addition, by requiring both $F_1$ and $F_2$ to be convex and supposing they are positively homogeneous of degree one, we produce quadratic growth for the solutions away from the free boundary. This result follows from a dyadic analysis, combined with the maximum principle. The argument relies on a scaling strategy, using the $L^\infty$-norms of the solutions as a normalization factor. This machinery was introduced in \cite{Caffarelli-Karp-Shahgholian2000} in the context of an obstacle problem driven by the Laplacian. In \cite{Lee-Shahgholian2001} the authors took this perspective to the fully nonlinear setting and developed a fairly complete analysis of the obstacle problem governed by fully nonlinear equations. We also refer the reader to \cite{Lee1998}.

The quadratic growth results developed in \cite{Caffarelli-Karp-Shahgholian2000} and \cite{Lee-Shahgholian2001} rely on a smallness condition on the density of the region where solutions are negative. Our argument resorts to a similar assumption. In fact, we consider the quantity
\begin{equation*}\label{eq_volume0}
	V_r(x^*,u) := \frac{\vol\left(B_r(x^*)\cap \Omega^-(u)\right)}{r^d};
\end{equation*}
by supposing $V_r(x,u)$ is controlled for every $x\in\partial(\Omega^+(u)\cup\Omega^-(u))\cap B_{1/2}$, we are capable of proving quadratic growth for the solutions, away from the free boundary. A further scaling argument -- depending on the square of the distance to the free boundary -- is capable of relating $B_1$ with each connected component associated with the transmission problem. This fact extrapolates regularity information for $u$; namely, we prove that strong solutions to \eqref{eq_main} are of class $C ^{1,1}$ in $B_{1/2}$. This is the content of our first main result.

\begin{teo}[Regularity of the solutions]\label{teo_main1}
Let $u\in W_{\rm loc}^{2,d}(B_1)$ be a strong solution to \eqref{eq_main}. Suppose A\ref{assump_ellipticity}-A\ref{assump_smalldensity}, to be detailed further, hold true. Then, $u\in C ^{1,1}_{\rm loc}(B_1)$ and there exists a universal constant $C>0$ such that 
\[
	\left\|D^2u\right\|_{L^\infty(B_{1/2})} \leq C.
\]	
\end{teo}

We observe that $F_1$ and $F_2$ are not required to be close, or even comparable, in any topoloy. After examining the regularity of the solutions, we turn our attention to the free transmission interface. Set $\Omega:=\Omega^+(u)\cup\Omega^-(u)$. At this point, an alternative arises. In fact, we can either suppose $\{Du\neq0\}\subset\Omega$ or $\{Du\neq0\}\not\subset\Omega$. In the former case, we are capable of producing a characterization of the global solutions to \eqref{eq_main}. 

We start our analysis by supposing $\{Du\neq0\}\subset\Omega$ and establishing a non-degeneracy result. Very much based on the maximum principle, it follows along the same lines put forward in \cite{Lee-Shahgholian2001} and \cite{Figalli-Shahgholian_2014}. The non-degeneracy property combines with Theorem \ref{teo_main1} to control quadratically the growth of the solutions from above and from below.

A further consequence of non-degeneracy concerns global solutions to \eqref{eq_main}; it relies on a condition concerning the thickness of the free boundary. For $A\subset\mathbb{R}^d$, denote with ${\rm MD}(A)$ the smallest distance between two hyperplanes enclosing $A$. Our second main result reads as follows.

\begin{teo}[Characterization of global solutions]\label{prop_hspac}
Let $u\in W_{\rm loc}^{2,d}(\mathbb{R}^d)$ be a strong solution to \eqref{eq_main} in $\mathbb{R}^d$. Suppose A\ref{assump_ellipticity}-A\ref{assump_smalldensity}, to be detailed below, are in force. Suppose further there exists $\varepsilon_0>0$, to be determined, such that
\[
	\frac{{\rm MD}\left((B_1\setminus\Omega)\cap B_r(x)\right)}{r} > \varepsilon_0,
\]
for $0<r\ll 1$ and $x\in\partial\Omega$. Then $u$ is a half-space solution. That is, up to a rotation, 
\[
	u(x) = \frac{\gamma[(x_1)_+]^2}{2} + C,
\]
where $C\in\mathbb{R}$ and $\gamma \in (1/\Lambda, 1/\lambda)$ is such that either $F_1(\gamma e_1\otimes e_1) = 1$ or $F_2(\gamma e_1\otimes e_1)=1$.
\end{teo}

%


We mention that a complete result on the regularity of the free boundary as well as the analysis of singular sets are off the scope of this paper; we refer the reader to \cite{Lee-Shahgholian2001} for an account of the free boundary regularity in the fully nonlinear setting. For the analysis of the singular set in the context of fully nonlinear problems, we mention \cite{Savin-Yu2019}.

The remainder of this paper is organized as follows: Section \ref{sec_mafp} gathers elementary results and details the main assumptions under which we work. In Section \ref{sec_regsol} we study the regularity of the strong solutions to \eqref{eq_main} and present the proof of Theorem \ref{teo_main1}. A fourth section reports a preliminary analysis of the free boundary, establishing the non-degeneracy property and proving Theorem \ref{prop_hspac}. 

\section{Preliminaries}\label{sec_mafp}

This section presents some preliminary material, as well as the main hypotheses we use in the paper. With $\mathcal{S}(d)$ we denote the space of symmetric matrices of order $d$; when convenient, we identify  $\mathcal{S}(d)\sim\mathbb{R}^\frac{d(d+1)}{2}$. We start with the uniform ellipticity of the operators $F_i$.

\begin{Assumption}[Uniform ellipticity]\label{assump_ellipticity}
For $i=1, 2$, we suppose the operator $F_i:\mathcal{S}(d)\to \mathbb{R}$  to be $(\lambda, \Lambda)$-uniformly elliptic. That is, for $0<\lambda\leq \Lambda$, it holds
\[
	\lambda\|N\| \leq F_i(M + N) - F_i(M) \leq \Lambda\|N\|,
\]
for every $M, N\in\mathcal{S}(d)$, $N\geq 0$, and $i=1, 2$. We also suppose $F_i(0)=0$.
\end{Assumption}

Uniform ellipticity relates closely with the extremal operators
\[
	\mathcal{M}^+_{\lambda,\Lambda}(M):=\Lambda\sum_{e_i>0}e_i+\lambda\sum_{e_i<0}e_i
\]
and
\[
	\mathcal{M}^-_{\lambda,\Lambda}(M):=\lambda\sum_{e_i>0}e_i+\Lambda\sum_{e_i<0}e_i.
\]
In fact, Assumption A\ref{assump_ellipticity} can be rephrased as
\[
	\mathcal{M}_{\lambda,\Lambda}^-(M-N)\leq F_i(M)-F_i(N)\leq\mathcal{M}^+_{\lambda,\Lambda}(M-N),
\]
for every $M,N\in\mathcal{S}(d)$, and $i=1,2$. For completeness, we recall the definition of viscosity solutions.

\begin{Definition}[$C$-viscosity solution]\label{def_viscsol}Let $G:\mathcal{S}(d)\to\mathbb{R}$ be a $(\lambda,\Lambda)$-elliptic operator.  We say that $u\in{\rm USC}(B_1)$ is a $C$-viscosity subsolution to 
\begin{equation}\label{eq_defvisc}
	G(D^2u) = 0\hspace{.4in}\mbox{in}\hspace{.1in}B_1
\end{equation}
if, for every $\varphi\in C^{2}_{\rm loc}(B_1)$ and $x_0\in B_1$, such that $u-\varphi$ attains a local maximum at $x_0$, we have
\[
	G(D^2\varphi(x_0)) \leq 0.
\]
Similarly, we say that $u\in{\rm LSC}(B_1)$ is a $C$-viscosity supersolution to \eqref{eq_defvisc} if, for every $\varphi\in C^{2}_{\rm loc}(B_1)$ and $x_0\in B_1$, such that $u-\varphi$ attains a local minimum at $x_0$, we have
\[
	G(D^2\varphi(x_0)) \geq 0.
\]
If $u\in C(B_1)$ is simultaneously a subsolution and a supersolution to \eqref{eq_defvisc}, we say it is a viscosity solution to the equation.
\end{Definition}

For $0<\lambda\leq\Lambda$ and $f\in C(B_1)$, we define $\overline{S}(\lambda,\Lambda,f)$ as the set of functions $u\in C(B_1)$ satisfying
\[
	\mathcal{M}^-_{\lambda,\Lambda}(D^2u)\leq f
\]
in $B_1$, in the viscosity sense. Similarly, $\underline{S}(\lambda,\Lambda,f)$ is the set of functions $u\in C(B_1)$ satisfying
\[
	\mathcal{M}^+_{\lambda,\Lambda}(D^2u)\geq f.
\]
Finally, we set
\[
	S(\lambda,\Lambda,f):=\overline{S}(\lambda,\Lambda,f)\cap\underline{S}(\lambda,\Lambda,f)
\]	
and
\[
	S^*(\lambda,\Lambda,f):=\overline{S}(\lambda,\Lambda,-|f|)\cap\underline{S}(\lambda,\Lambda,|f|).
\]	

For a comprehensive account of the theory of $C$-viscosity solutions, we refer the reader to \cite{CafCab}. When examining global solutions to \eqref{eq_main}, we also resort to the notion of $L^d$-viscosity solution; see \cite{CCKS96}. We proceed with the definition of strong solution.

\begin{Definition}[$W^{2,d}$-strong solution]\label{def_strongsol}
We say that $u\in W^{2,d}_{\rm loc}(B_1)$ is a strong solution to
\[
	G(D^2u(x)) = 0\hspace{.4in}\mbox{in}\hspace{.1in}B_1
\]
if $u$ satisfies the equation at almost every $x\in B_1$.
\end{Definition}

We refer the reader to \cite[Chapter 9]{GilbargTrudinger} for further details on this class of solutions and their properties. Next we present an example.

\begin{example}[Radial solutions]
Consider the function $v$ given by
\[
	v(x) := \frac{|x|}{2d}^2 - \frac{1}{8d}
\]
restricted to $B_1$. We have $\Omega^+(v)=\{x\in B_1, |x|> 1/2\}$ and $\Omega^-(v)=\{x\in B_1, |x|< 1/2\}$. If we set $F_1=F_2:=\Delta$, then $v$ solves \eqref{eq_main}. In addition, the free transmission is given by $\{|x|=1/2\}$.
\end{example}

In the sequel, we put forward an assumption requiring both operators $F_1$ and $F_2$ to be convex.

\begin{Assumption}[Convexity]\label{assump_convexity}
For $i=1, 2$, we suppose the operator $F_i:\mathcal{S}(d)\to \mathbb{R}$  to be convex.
\end{Assumption}




The next assumption concerns homogeneity of degree $1$. It plays a major role in the quadratic growth of the solutions. The argument towards quadratic growth  in \cite{Caffarelli-Karp-Shahgholian2000} uses the linearity of the Laplacian operator. In \cite{Lee-Shahgholian2001} the authors notice that in the fully nonlinear case the condition that parallels linearity is the homogeneity of degree $1$. 

\begin{Assumption}[Homogeneity of degree one]\label{assump_homo}
We suppose $F_1$ and $F_2$ to be homogeneous of degree one; that is, for every $\tau\in\mathbb{R}$ and $M\in\mathcal{S}(d)$, we have
\[
	F_i(\tau  M) = \tau F_i(M),
\]
for every $i=1,2$.
\end{Assumption}

Before proceeding with further assumptions, we introduce some notation used throughout the paper. We denote by $\Omega^+(u)$ the subset of the unit ball where $u> 0$, whereas $\Omega^-(u)$ stands for the set where $u<0$. That is,
\[
	\Omega^+(u) := \left\{x\in B_1 \mid u(x) > 0\right\}\hspace{.2in}\mbox{and}\hspace{.2in}\Omega^-(u) := \left\{x\in B_1 \mid u(x) < 0\right\}.
\]
When referring to the set where $u\neq 0$ it is convenient to use the notation $\Omega(u):=\Omega^+(u)\cup\Omega^-(u)$. With $\partial\Omega(u)$ we denote the union of the topological boundaries of $\Omega^+$ and $\Omega^-$. I.e.,
\[
	\partial\Omega(u) := \left(\partial\Omega^+(u) \cup\partial\Omega^-(u)\right) \cap  B_1.
\]
Also, we denote with $\Sigma(u)$ the set where $u$ vanishes:
\[
	\Sigma(u) = \left\{x\in B_1 \mid u(x)=0\right\}.
\] 

A further condition imposed on the problem regards the subregion $\Omega^-(u)$. It is critical in proving quadratic growth of the solutions through the set of methods used in the paper. For $x^*\in\partial \Omega$ and $0<r\ll 1$, we consider the quantity
\begin{equation}\label{eq_volume}
	V_r(x^*,u) := \frac{\vol\left(B_r(x^*)\cap \Omega^-(u)\right)}{r^d}.
\end{equation}
For ease of notation, we set $V_r(0,u)=:V_r(u)$. 

\begin{Assumption}[Normalized volume of $\Omega^-(u)$]\label{assump_smalldensity}
We suppose there exists $C_0>0$, to be determined later, such that 
\[
	V_r(x^*,u) \leq C_0
\]
for every $x^*\in \partial\Omega(u)$ and every $r\in(0,1/2)$.
\end{Assumption}
The former assumption imposes a control on the size of the subregion where $u$ is negative. It resonates on the geometry of the free boundary; see Figure \ref{fig_geometry}. 

\bigskip

\begin{figure}[h]
\centering

\psscalebox{.60 .60} 
{
\begin{pspicture}(0,-3.1284568)(18.5,3.1284568)
\definecolor{colour0}{rgb}{0.8,0.8,0.8}
\pscircle[linecolor=black, linewidth=0.04, dimen=outer](15.5,-0.09154316){3.0}
\psbezier[linecolor=black, linewidth=0.002, fillstyle=crosshatch, hatchwidth=0.02, hatchangle=0.0, hatchsep=0.1](12.970633,1.4611151)(12.880956,1.2453679)(12.48801,0.71934164)(12.530127,-0.1490115139152431)(12.572244,-1.0173646)(12.969116,-1.6890235)(12.975697,-1.6831887)(12.982276,-1.677354)(12.859482,-1.7566767)(13.68962,-1.2325559)(14.519758,-0.708435)(14.239495,-0.881016)(14.241519,-0.8831887)(14.243543,-0.88536143)(14.06253,-0.4654618)(14.013671,-0.1540748)(13.964812,0.15731218)(14.243678,0.64595544)(14.236456,0.65605175)(14.229233,0.6661481)(14.257428,0.6448469)(13.583291,1.0813682)(12.909154,1.5178895)(13.815536,0.9292136)(12.9655695,1.4560518)
\pscircle[linecolor=black, linewidth=0.04, fillstyle=solid,fillcolor=colour0, dimen=outer](15.5,-0.09154316){1.5}
\psbezier[linecolor=black, linewidth=0.002, fillstyle=solid](14.236456,0.67630494)(14.231597,0.6862829)(14.095536,0.39288762)(14.064303,0.1902289924138711)(14.0330715,-0.012429634)(14.02719,-0.24362262)(14.07443,-0.42242923)(14.121671,-0.60123587)(14.246002,-0.86701995)(14.241519,-0.86799884)(14.237036,-0.8689777)(15.477188,-0.09073465)(15.482025,-0.09331531)(15.486862,-0.09589597)(15.107416,0.12028507)(14.256709,0.6459252)
\pscircle[linecolor=black, linewidth=0.04, fillstyle=solid,fillcolor=colour0, dimen=outer](5.66,-0.011543159){1.5}
\psbezier[linecolor=black, linewidth=0.002, fillstyle=solid](4.1977777,0.12623462)(4.1968336,0.12796025)(4.1991167,-0.16414799)(4.193333,-0.1715431594848633)(4.1875496,-0.17893834)(3.9077222,-0.19780011)(4.6822224,-0.08709872)(5.4567223,0.023602687)(5.4395037,-0.015469238)(5.437778,-0.015987605)(5.436052,-0.01650597)(5.0840387,0.016403876)(4.668889,0.06845684)(4.253739,0.1205098)(4.2125635,0.11561627)(4.211111,0.13067906)
\psbezier[linecolor=black, linewidth=0.04](0.86,2.9884567)(0.86,-0.011543159)(5.66,-0.011543159)(5.66,-0.01154315948486328)(5.66,-0.011543159)(0.86,-0.011543159)(0.86,-3.0115433)
\pscircle[linecolor=black, linewidth=0.04, dimen=outer](5.66,-0.011543159){3.0}
\psdots[linecolor=black, dotsize=0.1](5.637778,-0.011543159)
\psline[linecolor=black, linewidth=0.04, linestyle=dashed, dash=0.17638889cm 0.10583334cm](5.6466665,0.0062346184)(5.9933333,1.4106791)
\psbezier[linecolor=black, linewidth=0.002, fillstyle=crosshatch, hatchwidth=0.02, hatchangle=0.0, hatchsep=0.1](2.74,0.54401237)(2.753396,0.53395313)(2.7,0.56179017)(2.6777778,0.046234618292915)(2.6555555,-0.46932092)(2.7484496,-0.58702135)(2.74,-0.58932096)(2.7315505,-0.5916205)(3.2954895,-0.38465032)(3.4955556,-0.32265428)(3.6956215,-0.26065823)(4.1958985,-0.16898705)(4.188889,-0.16709872)(4.181879,-0.16521038)(4.1836395,0.11466468)(4.188889,0.12179017)(4.1941385,0.12891567)(3.4314184,0.30839092)(3.442222,0.29067907)(3.453026,0.27296722)(2.756215,0.5202161)(2.7533333,0.5262346)
\psline[linecolor=black, linewidth=0.04, linestyle=dashed, dash=0.17638889cm 0.10583334cm](5.6555557,-0.024876492)(8.326667,-1.3359876)
\psdots[linecolor=black, dotsize=0.1](15.4777775,-0.09154316)
\psline[linecolor=black, linewidth=0.04, linestyle=dashed, dash=0.17638889cm 0.10583334cm](15.486667,-0.07376538)(15.833333,1.330679)
\psline[linecolor=black, linewidth=0.04, linestyle=dashed, dash=0.17638889cm 0.10583334cm](15.495556,-0.104876496)(18.166666,-1.4159876)
\psline[linecolor=black, linewidth=0.04](10.692727,2.890275)(15.48,-0.09154316)(10.7,-3.1115432)
\rput[bl](15.14,-0.6715432){\Large{$x^*$}}
\rput[bl](5.24,-0.5515432){\Large{$x^*$}}
\rput[bl](7.56,-1.6515431){\Large{$r_1$}}
\rput[bl](17.46,-1.7315432){\Large{$r_1$}}
\rput[bl](5.3,0.9484568){\Large{$r_2$}}
\rput[bl](15.18,0.8884568){\Large{$r_2$}}
\rput[bl](0.26,-0.27154315){\LARGE{$u < 0$}}
\rput[bl](10.48,-0.33154315){\LARGE{$u < 0$}}
\rput[bl](0.0,2.548457){\LARGE{$\partial\Omega$}}
\rput[bl](11.24,2.7084568){\LARGE{$\partial\Omega$}}
\end{pspicture}
}
\caption{The geometry depicted on the left is within the scope of \eqref{eq_smalldensity}. In fact, as the radii of the balls centered at $x^*$ decrease from $r_1$ to $r_2$, $V(x^*,r)$ decreases even faster. The case on the right behaves differently. Here, the normalized volume is constant, independent of the radii of the ball; hence, it might fail to satisfy a prescribed smallness regime as in \eqref{eq_smalldensity}.}\label{fig_geometry}
\end{figure}
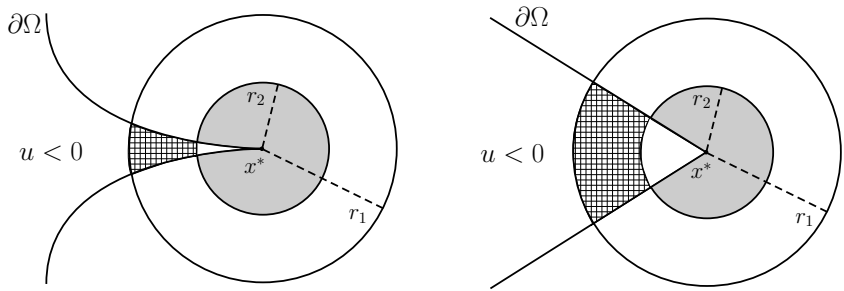

\bigskip

%
%

We proceed by introducing the notion of thickness. For any set $A$, we denote by $\text{MD}(A)$ the smallest possible distance between two parallel hyperplanes containing $A$. For a function $u\in W_{\rm loc}^{2,d}(\mathbb{R}^d)$ we define the thickness of $\Sigma(u)$ in $B_r(x)$ as
\begin{equation}\label{eq_thicknessdef}
	\delta_r(u,x)  :=  \frac{\text{MD}(\Sigma(u)\cap B_r(x))}{r}.
\end{equation}
The thickness $\delta_r$ satisfies some properties which we list below. We refer to \cite[Chapter 5]{Petrosyan-Shahgholian-Uraltseva2012} for more details. See also \cite{Figalli-Shahgholian_2014}.
\begin{Proposition}\label{prop_thick}
Let $u\in W_{\rm loc}^{2,d}(\mathbb{R}^d)$ be a solution to \eqref{eq_main} in $\mathbb{R}^d$. The measure of thickness $\delta_r$, introduced in \eqref{eq_thicknessdef}, satisfies the following properties:
\begin{enumerate}
\item Let $u_r:B_1\to\mathbb{R}$ be defined as 
\[
	u_r(x) := \frac{u(rx)}{r^2},
\]
for $r>0$. Then, 
\[
	\delta_1(u_r,0) = \delta_r(u,x);
\]
\item Let $P_2:B_1\to\mathbb{R}$ be a polynomial global solution of the form
\[
	P_2(x) := \sum_j a_jx^2_j,
\]
with $a_j$ such that either $F_1(D^2P_2) = 1$ or $F_2(D^2P_2)=1$. Then, we have $\delta_r(P_2,0) = 0$;
\item If $u_r$ converges to some function $u_0$ then 
\[
\limsup_{r\to0}\delta_r(u, x_0) \leq \delta_1(u_0,0).
\]
\end{enumerate}
\end{Proposition}

In the next section we examine the regularity of strong solutions to \eqref{eq_main}. In particular, we present the proof of Theorem \ref{teo_main1}.

\section{Regularity of the solutions}\label{sec_regsol}

In this section we detail the proof of Theorem \ref{teo_main1}. We start by relating \eqref{eq_main} with viscosity inequalities of the form
\begin{equation}\label{eq_viscineq1}
	\min\left(F_1(D^2u),F_2(D^2u)\right)\leq 1\hspace{.4in}\mbox{in}\hspace{.1in}B_1
\end{equation}
and
\begin{equation}\label{eq_viscineq2}
	\max\left(F_1(D^2u),F_2(D^2u)\right)\geq -1\hspace{.4in}\mbox{in}\hspace{.1in}B_1
\end{equation}

\begin{Lemma}\label{prop_strongareldvisc}
Let $u\in W_{\rm loc}^{2,d}(B_1)$ be a strong solution to \eqref{eq_main}. Suppose A\ref{assump_ellipticity} holds true. Then $u$ is a $C$-viscosity solution to the inequalities \eqref{eq_viscineq1}-\eqref{eq_viscineq2}.
\end{Lemma}

The proof of Lemma \ref{prop_strongareldvisc} follows from standard computations and the maximum principle for $W^{2,d}$-functions, see \cite[Corollary 3]{Lions83} and \cite{Bony1967}. In addition, if $u$ is a continuous viscosity solution to \eqref{eq_viscineq1}-\eqref{eq_viscineq2} we also have $u\in \overline{S}(\lambda,\Lambda,1)\cap S^*(\lambda,\Lambda,1)$. In fact, because
\[
	\mathcal{M}^-_{\lambda,\Lambda}(M)\leq F_i(M)\leq \mathcal{M}^+_{\lambda,\Lambda}(M)
\]
holds for $i=1,2$, we have
\[	
	\mathcal{M}^-_{\lambda,\Lambda}(D^2u)\leq \min\left(F_1(D^2u),F_2(D^2u)\right)\leq 1
\]
and
\[	
	\mathcal{M}^+_{\lambda,\Lambda}(D^2u)\geq \max\left(F_1(D^2u),F_2(D^2u)\right)\geq -1.
\]
As a consequence to the former inclusion we derive the H\"older continuity for the strong solutions to \eqref{eq_main}, with universal estimates.

\begin{Lemma}[H\"older continuity]\label{lem_holdercont}
Let $u\in W^{2,d}_{\rm loc}(B_1)$ be a strong solution to \eqref{eq_main} and suppose A\ref{assump_ellipticity} holds. Then $u\in C^\alpha_{\rm loc}(B_1)$, for some $\alpha\in(0,1)$, and there exists $C>0$ such that
\[
	\left\|u\right\|_{C^\alpha(B_{1/2})}\leq C\left(\left\|u\right\|_{L^\infty(B_1)}+\left\|f\right\|_{L^d(B_1)}\right).
\] 
In addition, $\alpha=\alpha(\lambda,\Lambda,d)$ and $C=C(\lambda,\Lambda,d)$.
\end{Lemma}

For a proof of Lemma \ref{lem_holdercont}, see \cite[Lemma 4.10]{CafCab}. In the sequel we prove that solutions to \eqref{eq_main} satisfy a quadratic growth \emph{away from the free boundary}.

\subsection{Quadratic growth away from the free boundary}\label{subsec_qg1}

Let $x^*\in B_1$ be fixed. Consider the maximal subset of $\mathbb{N}$ whose elements $j$ are such that
\begin{equation}\label{eq_16ju}
	\sup_{x\in B_{2^{-j-1}}(x^*)}|u(x)| \geq \frac{1}{16}\sup_{x\in B_{2^{-j}}(x^*)}|u(x)|;
\end{equation}
we denote such set by $\mathcal{M}(x^*,u)$.

\begin{Proposition}\label{prop_quadratic}
Let $u\in W_{\rm loc}^{2,d}(B_1)$ be a strong solution to \eqref{eq_main}. Suppose A\ref{assump_ellipticity}-A\ref{assump_homo} hold true. Let $x^*\in\partial\Omega$. There exists $C_0>0$ such that, if
\begin{equation}\label{eq_smalldensity}
	V_{2^{-j}}(x^*,u) < C_0,
\end{equation}
for every $j\in \mathcal{M}(x^*,u)$, then
\[
	\sup_{x\in B_{2^{-j}}(x^*)} |u(x)| \leq \frac{1}{C_0}2^{-2j}, \hspace{.4in} \forall j \in \mathcal{M}(x^*,u).
\]
\end{Proposition}
\begin{proof}
For ease of presentation, we split the proof in three steps. 

\medskip

\noindent{\bf Step 1 - }Set $x^*=0$ and $\mathcal{M}(u):=\mathcal{M}(0,u)$. We resort to a contradiction argument; suppose the statement of the proposition is false. Then, there exist sequences $(u_n)_{n\in\mathbb{N}}$ and $(j_n)_{n\in\mathbb{N}}$ such that $u_n$ is a normalized strong solution to \eqref{eq_main}, 
\begin{equation}\label{eq_quadcont1}
	V_\frac{1}{2^{n}}(u_n) < \frac{1}{n},
\end{equation}
with
\begin{equation}\label{eq_quadcont2}
	\sup_{x\in B_{2^{-j_n}}} |u_n(x)| > \frac{n}{2^{2j_n}},
\end{equation}
for every $j_n\in\mathcal{M}(u_{n})$, and $n\in\mathbb{N}$. Because $\left\|u_n\right\|_{L^\infty(B_1)}$ is uniformly bounded, it follows from \eqref{eq_quadcont2} that $j_n\longrightarrow \infty$. In particular, we may re-write \eqref{eq_quadcont1} as
\begin{equation}\label{eq_quadcont1a}
	V_\frac{1}{2^{j_n}}(u_n) < \frac{1}{j_n}.
\end{equation}

\medskip

\noindent{\bf Step 2 - }Now, we introduce an auxiliary function $v_n:B_1\to\mathbb{R}$, given by
\[
	v_n(x) := \frac{u_n(2^{-j_n}x)}{\left\|u_n\right\|_{L^\infty\left(B_{2^{-(j_n+1)}}\right)}}.
\]
Clearly, $v_n(0)=0$. In addition, $V_1(v_n)\longrightarrow 0$. Moreover, it follows from the definition of $v_n$ that 
\begin{equation}\label{eq_quadcont3}
	\sup_{B_{1/2}}|v_n(x)| = 1
\end{equation}
and
\[
	\sup_{B_{1}}|v_n(x)| \leq 16.
\]
We notice that A\ref{assump_homo} yields
\[
	\min\left(F_1(D^2v_n),F_2(D^2v_n)\right)\leq\frac{\min\left(F_1(D^2u_n(2^{-j_n}x)),F_2(D^2u_n(2^{-j_n}x))\right)}{2^{2{j_n}}\left\|u_n\right\|_{L^\infty\left(B_{2^{-(j_n+1)}}\right)}}.
\]
Therefore,
\begin{equation}\label{eq_quadcont4}
	\min\left(F_1(D^2v_n),F_2(D^2v_n)\right)\leq \frac{1}{n} \frac{C\left\|u_n\right\|_{L^\infty\left(B_{2^{-j_n}}\right)}}{\left\|u_n\right\|_{L^\infty\left(B_{2^{-(j_n+1)}}\right)}} \leq \frac{C}{n}\leq C_0,
\end{equation}
for some $C_0>0$ and $n\gg 1$. On the other hand,
\begin{align}\nonumber
	\max\left(F_1(D^2v_n),F_2(D^2v_n)\right)&\geq \frac{\max\left(F_1(D^2u_n(\frac{x}{2^{j_n}})),F_2(D^2u_n(\frac{x}{2^{j_n}}))\right)}{2^{2{j_n}}\left\|u_n\right\|_{L^\infty\left(B_{2^{-(j_n+1)}}\right)}}\\\label{eq_vbound}
		&\geq -C_0.
\end{align}
It follows from \eqref{eq_quadcont4}-\eqref{eq_vbound} that $(v_n)_{n\in\mathbb{N}}\subset S^*(\lambda,\Lambda,C_0)$. As a consequence, $v_n\in C^\alpha_{loc}(B_1)$ for every $n\in \mathbb{N}$, for some unknown $\alpha\in (0,1)$, with uniform estimates; see \cite[Proposition 4.10]{CafCab}. Therefore, there exists $v_\infty$ such that $v_n\longrightarrow v_\infty$ in $C ^{\beta}_{\rm loc}(B_1)$, for every $0<\beta<\alpha$. Since $v_n(0)=0$ for every $n\in\mathbb{N}$ we infer that $v_\infty(0)=0$, whereas \eqref{eq_quadcont3} leads to $\left\|v_\infty\right\|_{L^\infty(B_{1/2})}=1$. Because $V_1(v_n)\longrightarrow 0$, we conclude that $v_\infty\geq 0$ in $B_1$.

\medskip

\noindent{\bf Step 3 - }Standard stability results for viscosity solutions build upon \eqref{eq_quadcont4} to ensure
\[
	\min\left(F_1(D^2v_\infty),F_2(D^2v_\infty)\right)\leq 0 \hspace{.4in} \mbox{in} \hspace{.4in} B_1.
\]
We conclude that $v_\infty\in \overline{S}(\lambda,\Lambda,0)$ attains an interior local minimum at the origin, which leads to a contradiction (see, for instance, \cite[Proposition 4.9]{CafCab}).
\end{proof}

In Proposition \ref{prop_quadratic} the constant $C_0>0$ informing A\ref{assump_smalldensity} is determined. This quantity remains unchanged throughout the paper. The next result extrapolates the former analysis from $\mathcal{M}(x^*,u)$ to the entire set of natural numbers.

\begin{Proposition}\label{prop_quadnaturals}
Let $u\in W_{\rm loc}^{2,d}(B_1)$ be a strong solution to \eqref{eq_main}. Suppose A\ref{assump_ellipticity}-A\ref{assump_homo} hold true. Let $x^*\in\partial\Omega$. Suppose further that for every $j\in \mathcal{M}(x^*,u)$ we have
\[
	V_{2^{-j}}(x^*,u) < C_0,
\]
for $C_0>0$ fixed in \eqref{eq_smalldensity}. Then
\[
	\sup_{x\in B_{2^{-j}}(x^*)} |u(x)| \leq \frac{4}{C_0}2^{-2j}, \hspace{.4in} \forall j \in \mathbb{N}.
\]
\end{Proposition}
\begin{proof}
As before we set $x^*=0$ and argue through a contradiction argument. Suppose the proposition is false. Let $m\in\mathbb{N}$ be the smallest natural number such that 
\begin{equation}\label{eq_quadcont5}
	\sup_{B_{2^{-m}}} |u(x)| > \frac{4}{C_0}2^{-2m}.
\end{equation}
We claim that $m-1\in\mathcal{M}(u)$. Indeed,
\[
	\sup_{B_{2^{1-m}}} |u(x)| \leq \frac{4}{C_0}2^{-2(m-1)} = \frac{16}{C_0}2^{-2m} < 4\sup_{B_{2^{-m}}} |u(x)|.
\]
We conclude
\[
	\sup_{B_{2^{-m}}} |u(x)| \leq \sup_{B_{2^{1-m}}} |u(x)| \leq \frac{1}{C_0}2^{-2(m-1)} = \frac{4}{C_0}2^{-2m},
\]
which contradicts \eqref{eq_quadcont5} and completes the proof.
\end{proof}

Consequential to Proposition \ref{prop_quadnaturals} is the quadratic growth of $u$ away from the free boundary. This is the content of the next corollary.

\begin{Corollary}[Quadratic growth]\label{cor_quadratic}
Let $u\in W_{\rm loc}^{2,d}(B_1)$ be a strong solution to \eqref{eq_main}. Suppose A\ref{assump_ellipticity}-A\ref{assump_homo} hold true. Let $x^*\in\partial\Omega\cap B_{1/2}$. Suppose further that, for every $j\in \mathcal{M}(x^*,u)$, we have
\[
	V_{2^{-j}}(x^*,u) < C_0,
\]
for $C_0>0$ as in Proposition \ref{prop_quadratic}. Then, for $0<r<1/2$ there exists $C>0$ such that
\[
	\sup_{x\in B_{r}(x^*)} |u(x)| \leq Cr^2,
\]
where $C=C\left(d,\lambda,\Lambda,\left\|u\right\|_{L^\infty(B_1)}\right)$.
\end{Corollary}
\begin{proof}
Find $j\in\mathbb{N}$ satisfying $2^{-(j+1)}\leq r< 2^{-j}$. It is straightforward to notice that
\[
	\begin{split}
		\sup_{B_r}|u(x)| \leq \sup_{B_{2^{-j}}}|u(x)| \leq C\left[\left(\frac{1}{2}\right)^{j+1-1}\right]^2 \leq Cr^2,
	\end{split}
\]	
which ends the proof.
\end{proof}

We close this section with the proof of Theorem \ref{teo_main1}; compare with \cite[Theorem 1.1]{Caffarelli-Karp-Shahgholian2000} and \cite[Remark 1.3]{Lee-Shahgholian2001}.

\begin{proof}[Proof of Theorem \ref{teo_main1}]
Suppose $0\in\partial\Omega$. Corollary \ref{cor_quadratic} leads to
\[
	|u(x)| \leq C\left[\dist(x,\partial\Omega)\right]^2,
\]
for every $x\in B_{1/2}$. Consider the auxiliary function $v:B_1\to\mathbb{R}$ given by
\[
	v(y) := \frac{u(x + y\dist(x,\partial\Omega))}{\left[\dist(x,\partial\Omega)\right]^2};
\]
clearly, $\left|D^2u(x)\right|=\left|D^2v(0)\right|$. Notice that
\[
	\left\{z\in B_1 \mid y\in B_1\;\;\;\mbox{and}\;\;\;z := x + y\dist(x,\partial\Omega)\right\}
\]
is contained in the same connected component to which $x$ belongs. Therefore, $F_i(D^2v)=1$ or $F_1(D^2v)=0$ in the unit ball. Hence, standard results in elliptic regularity theory produce
\[
	\left|D^2v(0)\right| \leq C,
\]
for some universal constant $C>0$, not depending on $x$, and the proof is complete.
\end{proof}

In the next section, we turn our attention to the analysis of the free interface. We start working under the assumption $\{Du\neq 0\}\subset \Omega$ and produce a characterization of global solutions.

\section{Classification of global solutions}

In this section we examine the non-degeneracy of the free boundary. In addition we study properties of the global solutions. We start with the non-degeneracy property. This is the content of the next proposition.

\begin{Proposition}[Non-degeneracy of the free boundary]\label{prop_non-deg}
Let $u\in W_{\rm loc}^{2,d}(B_1)$ be a strong solution to \eqref{eq_main}. Suppose that A\ref{assump_ellipticity}-A\ref{assump_smalldensity} are in force and $\{Du\neq 0\}\subset \Omega$. Let $x^*\in \partial\Omega\cap B_{1/2}$. There exists $C>0$ such that 
\[
	\sup_{x\in\partial B_r(x^*)} u(x) \geq  Cr^2
\]
for every $0<r<1/2$.
\end{Proposition}
\begin{proof}
We split the proof in three steps. 

\medskip

\noindent{\bf Step  1.} Without loss of generality, we take $x^*\in\Omega$. Furthermore, the set $\{Du\neq0\}$ is dense in $\Omega$. It follows from the fact that $D^2u(x)=0$ for almost every $x\in\{ Du=0\}$, the last condition in A\ref{assump_ellipticity} and \eqref{eq_main}. Therefore, we suppose $x\in\{Du\neq0\}\cap\Omega$.

\medskip

\noindent{\bf Step  2.} We introduce an auxiliary function $w\in W_{\rm loc}^{2,d}(B_1)$, given by
\[
	w(x) := u(x) - \frac{|x - x^*|^2}{2d\lambda}.
\]
We claim that
\begin{equation}\label{eq_maxw}
	\max_{x\in\partial B_r(x^*)} w(x) = \sup_{x\in B_r(x^*)} w(x),
\end{equation}
for every $0<r<1/2$. It follows from \eqref{eq_maxw} that
\[
	\max_{x\in B_r(x^*)} u(x) \geq u(x^*) + Cr^2,
\]
where $C:=4d\Lambda$. By approximation, the former inequality yields the results. It remains to establish \eqref{eq_maxw}.

\medskip

\noindent{\bf Step 3.} Suppose \eqref{eq_maxw} is false. There exists a maximum point $y\in B_r(x^*)$ for $w$. Hence,
\[
	Dw(y) = Du(y) - \frac{y - x^*}{d\Lambda} = 0.
\]
Were $y=x^*$, it would be $Du(x^*)=0$; we conclude that $y\neq x^*$ and, in addition, 
\[
	Du(y) \neq 0.
\]
By assumption, we have $y\in \Omega$.

On the other hand, $w$ is a solution for $\max(F_1(D^2w),F_2(D^2w))\geq0$ in $\Omega$. In fact, for $x\in \Omega$, 
\[
	\max(F_1(D^2w),F_2(D^2w)) \geq  1 - \mathcal{M}^+\left(\frac{I_d}{d\lambda}\right) \geq 0 .
\]
Since $y\in\Omega$, there exists a neighborhood of $y$ where $w$ is constant. We conclude $w$ is constant in $B_r(x^*)$ and \eqref{eq_maxw} follows.
\end{proof}

\begin{Remark}\normalfont
The proof of Proposition \ref{prop_non-deg} follows closely the ideas put forward in \cite[Lemma 3.1]{Figalli-Shahgholian_2014}.
\end{Remark}

The former result has a number of standard consequences. Of particular interest is the negligibility of the free boundary in the sense of Lebesgue.

\begin{Corollary}[Lebesgue negligibility of the free boudnary]\label{cor_lebesgueneg}
Let $u\in W_{\rm loc}^{2,d}(B_1)$ be a strong solution to \eqref{eq_main}. Suppose that A\ref{assump_ellipticity}-A\ref{assump_smalldensity} are in force and $\{Du\neq 0\}\subset \Omega$. Then $\partial\Omega$ has Lebesgue measure zero.
\end{Corollary}

Next, we establish the classification of global solutions. At this point, we derive a further ancillary fact from Theorem \ref{teo_main1}. 
\begin{Remark}\label{rem_pow}\normalfont
Let $u\in W^{2,d}_{\rm loc}(B_1)$ be a strong solution to \eqref{eq_main}. We claim there exists a convex $(\lambda,\Lambda)$-elliptic operator $G$ and a source term $g:B_1\to\mathbb{R}$, locally bounded in the $L^\infty$-topology, such that
\[
	G(D^2u)=g\hspace{.4in}\mbox{in}\hspace{.1in}B_1
\]
in the $L^d$-viscosity sense; see \cite[Definition 2.1]{CCKS96}. In fact, we have
\[
	F_1(D^2u(x))=
		\begin{cases}
			1&\hspace{.4in}\mbox{\rm a.e.}-x\in \{u>0\}		\\
			0&\hspace{.4in}\mbox{\rm a.e.}-x\in \{u=0\}		\\
			F_1(D^2u(x))&\hspace{.4in}\mbox{\rm a.e.}-x\in \{u<0\};
		\end{cases}
\]
therefore, ellipticity and the conclusion of Theorem \ref{teo_main1} build upon the maximum principle for $W^{2,d}$-functions to yield the claim with $G:=F_1$. 
\end{Remark}
In the remainder of this section, $G$ denotes the convex, uniformly elliptic operator governing the equation satisfied by $u$ in the entire unit ball, in the $L^d$-sense.

For completeness, we recall the definition of thickness $\delta_r(u,x_0)$, stated in Section \ref{sec_mafp}: let $u\in W_{\rm loc}^{2,d}(B_1)$ be a strong solution to \eqref{eq_main} and suppose $x_0\in\partial\Omega(u)$. We have,
\[
	\delta_r(u,x_0) := \frac{\text{MD}\left(\Sigma(u)\cap B_r(x_0)\right)}{r}.
\]
We proceed with a proposition on the geometry of $u$.

\begin{Proposition}\label{prop_conv}
Let $u\in W_{\rm loc}^{2,d}(\mathbb{R}^d)$ be a strong solution to \eqref{eq_main} in $\mathbb{R}^d$. Suppose A\ref{assump_ellipticity}-A\ref{assump_smalldensity} are in force. Suppose further that there exists $\varepsilon_0>0$ such that
\begin{equation}\label{eq_thick}
\delta_r(u,x_0)  \geq  \varepsilon_0 
\end{equation} 
for all $r>0$ and every $x_0 \in \partial\Omega$. Then u is a convex function.
\end{Proposition}

\begin{proof}
We argue by contradiction and split the proof into a few steps.

\medskip

\noindent{\bf Step 1 - } Suppose that $u$ is not convex and define 
\begin{equation}\label{eq_min}
	-m :=  \inf_{z\in\Omega,  e \in \mathbb{S}^{d-1}}\partial_{ee}u(z)< 0.
\end{equation}
We claim $m>0$ is finite. In fact, because of Theorem \ref{teo_main1}, we have $u\in C ^{1,1}(\mathbb{R}^d)$. Now, consider a minimizing sequence $(y_n,e_n)\subset \Omega\times\mathbb{S}^{d-1}$ for \eqref{eq_min}; that is,
\[
	\partial_{e_ne_n}u(y_n)  \longrightarrow  -m 
\]
as $n\to\infty$. For $x\in B_1$, define the rescaled function 
\[
	u_n(x)  :=  \frac{u(d_nx + y_n)  -  u(y_n)  -  d_nDu(y_n)\cdot x}{d_n^2},	
\] 
where $d_n := \dist(y_n, \partial\Omega)$. Define further
\[
	\Omega_n :=  \frac{\Omega  -  y_n}{d_n}\hspace{.4in}\mbox{and}\hspace{.4in}\ell_n  := -\frac{Du(y_n)}{d_n}.
\]
Since $Du=0$ on $\partial\Omega$, we conclude $Du_n = \ell_n$ on $\partial\Omega_n$.

\medskip

\noindent{\bf Step 2 - }Now, observe that given $y_n \in \Omega$, we either have $y_n \in \Omega^+$ or $y_n \in \Omega^-$. It follows that $z = d_nx + y_n$ also belongs either to $\Omega^+$ or $\Omega^-$. Therefore, we have 
\[
	F_1(D^2u_n) = 1\hspace{.4in}\mbox{or}\hspace{.4in}F_2(D^2u_n) = 1
\]
in $\Omega_n$. In any case $u_n$ is $C^{2,\alpha}$-regular; hence, $\ell_n < C$ for every $n\in\mathbb{N}$ and some $C>0$, universal. As a consequence, we have $\ell_n \to \ell_\infty$, through some subsequence if necessary. 

Without loss of generality, suppose $e_n \to e_1$, as $n\to \infty$, through a subsequence, if required. Since $(u_j)_{n\in\mathbb{N}}$ is uniformly bounded in $C ^{2,\alpha}(B_{1/2})$ there exists $u_\infty \in C ^{2,\beta}(B_{1/2})$ such that $u_n\to u_\infty$ in the $C ^{2,\beta}$-topology, for $0<\beta<\alpha$. Moreover, $\partial_{11}u_\infty(0) = -m$. 

\medskip

\noindent{\bf Step 3 - } Let $G$ be the operator defined in Remark \ref{rem_pow}. Because $G$ is convex, $\partial_{11}u_{\infty}$ is a supersolution of the equation driven by the linear operator $G_{ij}(D^2u_\infty)\partial_{ij}$. Let $\Omega_\infty$ be the connected component containing $B_1$. Since $\partial_{11}u_\infty(z) \geq -m$ in $B_1$, the strong maximum principle yields $\partial_{11}u_\infty \equiv -m$ in $\Omega_\infty$.

Without loss of generality we can assume $Du_{\infty}(x) = 0$ on $\partial\Omega_\infty$; indeed, it follows from an affine transformation of $u_\infty$. For any $e \in \mathbb{S}^{d-1}$ we have $\partial_{ee}u_\infty(z) \geq -m$ in $B_1$; in addition,  the directional Hessian along $e_1$ attains $-m$. We conclude $e_1$ is an eigenvector for $D^2u$ at every point, associated with the smallest eigenvalue. It follows that $\partial_{1j}u_\infty = 0$ along $\partial \Omega_\infty$, for any $j=2, \ldots, d$. Integrating $u_\infty$ in the direction of $e_1$ we deduce
\[
	u_\infty(x)  =  P(x) :=  -m\frac{x_1^2}{2} + ax_1 + b(x') \hspace{0.4cm} \mbox{ in } \hspace{0.2cm}\Omega_\infty,
\]   
where  $x' = (x_2, \ldots,  x_d)$, $a\in\mathbb{R}$ is a fixed constant and $b:\mathbb{R}^{d-1}\to\mathbb{R}$. 

Observe that 
\[
	\frac{\partial}{\partial x_1}P(x) = -mx_1 + a;
\]
hence, $\partial_1 P$ vanishes along the set $\{x_1 = a/m\}$. On the other hand, the fact that $Du_\infty = 0$ on $\partial\Omega_\infty$ yields $\partial_1u_\infty=\partial_1P=0$ on $\partial\Omega_\infty$. As a consequence, we infer $\partial\Omega_\infty \subset \{x_1 = a/m\}$. 

\medskip

\noindent{\bf Step 4 - }At this point we distinguish two cases related to the former inclusion.

\medskip

\noindent{\bf Case 1, $\partial\Omega_\infty \neq \{x_1=a/m\}$ - }It follows that $\mathbb{R}^d\backslash\{x_1=a/m\} \subset \Omega_\infty$, and a further alternative is available, i.e.:
\[
	F_1(D^2u_\infty) = 1\hspace{.4in}\mbox{or}\hspace{.4in}F_2(D^2u_\infty) = 1
\]
almost everywhere in $\mathbb{R}^d$. The Evans-Krylov Theorem applies to $u_r(y) := u_\infty(ry)/r^2$ inside $B_1$ to produce
\[
\sup_{x,z \in B_r}\dfrac{|D^2u_\infty(x) - D^2u_\infty(z)|}{|x-z|^\alpha} \leq \dfrac{C}{r^\alpha}.
\]
Letting $r\to\infty$ we deduce that $D^2u_\infty$ is constant. Hence $u_\infty(x)$ is a second order polynomial.

\bigskip

\noindent{\bf Case 2, $\partial\Omega_\infty = \{x_1 = a/m\}$ - }In this case, we have $D_{x'}P = 0$ on $\{x_1 = a/m\}$ (recall that $Du_\infty = 0$ on $\partial\Omega_\infty$). Thus, $b$ is constant and we obtain
\[
u_\infty(x)  =  -m\dfrac{x_1^2}{2}  +  ax_1  +  b \hspace{0.4cm} \mbox{ in }\hspace{0.4cm} \{x_1>a/m\},
\]
which implies $D^2u_\infty \equiv - m\Id$. Being negative-definite, $D^2u_\infty$ cannot satisfy either $F_1(D^2u_\infty) = 1$ or $F_2(D^2u_\infty) = 1$, which leads to a contradiction.

Therefore, if u is not convex, it has to be a second order polynomial. By combining \eqref{eq_thick} and Proposition \ref{prop_thick} we conclude that $u_\infty$ cannot be a second degree polynomial, obtaining a contradiction and completing the proof.
\end{proof}

\begin{Corollary}
Let $u\in W_{\rm loc}^{2,d}(\mathbb{R}^d)$ be a strong solution to \eqref{eq_main} in $\mathbb{R}^d$. Suppose A\ref{assump_ellipticity}-A\ref{assump_smalldensity} hold and $\{Du\neq 0\}\subset \Omega$. Suppose further \eqref{eq_thick} is in force.
Then $\Omega = \{Du\not=0\}$.
\end{Corollary}

\begin{proof}
Because $u$ is convex, its set of critical points coincide with its set of minima; in addition, the set of minima of a convex function is convex. Hence, $\{Du=0\}$ is convex. Since $F_1(D^2u) = 1 \mbox{ in } \Omega^+$ and $F_2(D^2u) = 1 \mbox{ in } \Omega^-$ we have that $|\Omega\backslash\{Du\neq0\}| = 0$. As a consequence, the convex set $\{Du=0\}$ has measure zero in $\Omega$; if the former is nonempty, it must have co-dimension 1 and, therefore, violates the thickness condition \eqref{eq_thick}. Hence, $\Omega = \{Du\not=0\}$.
\end{proof}

In what follows, we produce the proof of Theorem \ref{prop_hspac}.

\begin{proof}[Proof of Theorem \ref{prop_hspac}.]
For simplicity we suppose that $0\in \partial\Omega$. For $r>0$ define 
\[
	u_r(x) := \frac{u(rx)}{r^2}
\]
and let $u_\infty$ be the limit, up to a sequence if necessary, of $u_r$ as $r\to\infty$. Notice that
\[
\Sigma(u_\infty) = \left\lbrace\Sigma(u):tx \in \Sigma(u) \quad\forall t>0\right\rbrace.
\]  
Next, we prove that $\Sigma(u_\infty)$ is a half-space. As before, we resort to a contradiction argument. Suppose $\Sigma(u_\infty)$ is a not a half-space; then in some system of coordinates, we have
\[
	\Sigma(u_\infty) \subset {C }_{\theta_0} := \{x\in\mathbb{R}^d;x=(\rho\cos\theta, \rho\sin\theta, x_3, \ldots, x_d), \theta_0\leq |\theta|\leq \pi\},
\]
for some $\theta_0>\pi/2$.

Choose $\theta_1 \in (\pi/2, \theta_0)$ and set $\alpha:=\pi/\theta_1$. Then for $\beta >0$ sufficiently large, the function 
\[
	v  :=  r^\alpha(e^{-\beta\sin(\alpha\theta)}  -  e^{-\beta})
\]
is a positive subsolution for the linear operator $G_{ij}(D^2u)\partial_{ij}$ inside $\mathbb{R}^d\backslash{C }_1$,  which vanishes on $\partial{C }_{\theta_1}$. From Proposition \ref{prop_conv} we have that  $u_\infty$ is convex; thus we deduce that $\partial_1u_\infty >0$ in $\mathbb{R}^d\backslash{C }_{\theta_0}$. In addition $\theta_0 > \theta_1$. Hence, by the comparison principle we obtain that
\[
v \leq \partial_1u_{\infty}.
\]
Therefore, $\Sigma(u_\infty)$ is a half-space that happens to be convex. As a consequence, it follows that $\Sigma(u)$ is also a half space.

Finally, we apply global $C^{2,\alpha}$-estimates to $u$ inside the half-ball $B_1\backslash\Sigma(u)$; see, for instance, \cite{Silvestre-Sirakov2014}. We obtain
\[
\sup_{x,z\in B_r\backslash\Sigma(u)}\dfrac{|D^2u(x)-D^2u(z)|}{|x-z|^\alpha} \leq \dfrac{C}{r^\alpha}.
\]
Thus, letting $r\to\infty$ we conclude that $D^2u$ is constant and hence $u$ is a second order polynomial inside the half-space $\mathbb{R}^d\backslash\Sigma(u)$. Recall that $Du=0$ on the hyperplane $\partial\Sigma(u)$, because we have supposed $\{Du\neq0\}\subset\Omega$. Hence, we conclude $u$ is a half-space solution and complete the proof.
\end{proof}

\bigskip

\noindent{\bf Acknowledgements.  }EP is partially funded by FAPERJ (E26/200.002/2018), CNPq-Brazil (\#433623/2018-7 and \#307500/2017-9) and Instituto Serrapilheira (\#1811-25904). MS is partially supported by PUC-Rio Arquimedes Fund. This study was financed in part by the Coordena\c{c}\~ao de Aperfei\c{c}oamento de Pessoal de N\'ivel Superior -- Brasil (CAPES) -- Finance Code 001

\bigskip

\noindent\textsc{Edgard A. Pimentel}\\
Department of Mathematics\\
Pontifical Catholic University of Rio de Janeiro -- PUC-Rio\\
22451-900, G\'avea, Rio de Janeiro-RJ, Brazil\\
\noindent\texttt{pimentel@puc-rio.br}

\bigskip

\noindent\textsc{Makson S. Santos}\\
Center of Investigations in Mathematics (CIMAT)\\
36000 Guanajuato Gto.MEXICO\\
\noindent\texttt{makson.santos@cimat.mx}


\begin{thebibliography}{10}

\bibitem{Amaral-Teixeira2015}
M. Amaral and E. Teixeira.
\newblock Free transmission problems.
\newblock {\em Comm. Math. Phys.}, 337(3):1465--1489, 2015.

\bibitem{Bao-Li-Yin2009}
E.~S. Bao, Y. Li, and B. Yin.
\newblock Gradient estimates for the perfect conductivity problem.
\newblock {\em Arch. Ration. Mech. Anal.}, 193(1):195--226, 2009.

\bibitem{Bao-Li-Yin2010}
E.~S. Bao, Y. Li, and B. Yin.
\newblock Gradient estimates for the perfect and insulated conductivity
  problems with multiple inclusions.
\newblock {\em Comm. Partial Differential Equations}, 35(11):1982--2006, 2010.

\bibitem{Bonnetier2000}
E. Bonnetier and M. Vogelius.
\newblock An elliptic regularity result for a composite medium with
  ``touching'' fibers of circular cross-section.
\newblock {\em SIAM J. Math. Anal.}, 31(3):651--677, 2000.

\bibitem{Bony1967}
J-M. Bony.
\newblock Principe du maximum dans les espaces de {S}obolev.
\newblock {\em C. R. Acad. Sci. Paris S\'{e}r. A-B}, 265:A333--A336, 1967.

\bibitem{Borsuk1968}
M. Borsuk.
\newblock A priori estimates and solvability of second order quasilinear
  elliptic equations in a composite domain with nonlinear boundary condition
  and conjugacy condition.
\newblock {\em Trudy Mat. Inst. Steklov.}, 103:15--50. (loose errata), 1968.

\bibitem{Borsuk2010}
M. Borsuk.
\newblock {\em Transmission problems for elliptic second-order equations in
  non-smooth domains}.
\newblock Frontiers in Mathematics. Birkh\"{a}user/Springer Basel AG, Basel,
  2010.

\bibitem{Briane-Capdeboscq-Nguyen2013}
M. Briane, Y. Capdeboscq, and L. Nguyen.
\newblock Interior regularity estimates in high conductivity homogenization and
  application.
\newblock {\em Arch. Ration. Mech. Anal.}, 207(1):75--137, 2013.

\bibitem{CafCab}
L. Caffarelli and X. Cabr\'e.
\newblock {\em Fully nonlinear elliptic equations}.
\newblock Colloquium Publications 43. AMS, Providence, 1995.

\bibitem{CCKS96}
L.~Caffarelli, M.~G. Crandall, M.~Kocan, and A.~\'Swi\polhk{e}ch.
\newblock On viscosity solutions of fully nonlinear equations with measurable
  ingredients.
\newblock {\em Comm. Pure Appl. Math.}, 49(4):365--397, 1996.

\bibitem{Caffarelli-Carro-Stinga2020}
L. Caffarelli, M.~Soria-Carro, and P.~R. Stinga.
\newblock Regularity for $C^{1,\alpha}$ interface transmission problems, 2020.

\bibitem{Caffarelli-Huang2003}
L. Caffarelli and Q. Huang.
\newblock Estimates in the generalized {C}ampanato-{J}ohn-{N}irenberg spaces
  for fully nonlinear elliptic equations.
\newblock {\em Duke Math. J.}, 118(1):1--17, 2003.

\bibitem{Caffarelli-Karp-Shahgholian2000}
L. Caffarelli, L. Karp, and H. Shahgholian.
\newblock Regularity of a free boundary with application to the {P}ompeiu
  problem.
\newblock {\em Ann. of Math. (2)}, 151(1):269--292, 2000.

\bibitem{Campanato1957}
S. Campanato.
\newblock Sul problema di {M}. {P}icone relativo all'equilibrio di un corpo
  elastico incastrato.
\newblock {\em Ricerche Mat.}, 6:125--149, 1957.

\bibitem{Campanato1959}
S. Campanato.
\newblock Sui problemi al contorno per sistemi di equazioni differenziali
  lineari del tipo dell'elasticit\`a. {I}.
\newblock {\em Ann. Scuola Norm. Sup. Pisa Cl. Sci. (3)}, 13:223--258, 1959.

\bibitem{Campanato1959a}
S. Campanato.
\newblock Sui problemi al contorno per sistemi di equazioni differenziali
  lineari del tipo dell'elasticit\`a. {II}.
\newblock {\em Ann. Scuola Norm. Sup. Pisa Cl. Sci. (3)}, 13:275--302, 1959.

\bibitem{Eldering2013}
J.~Eldering.
\newblock {\em Normally hyperbolic invariant manifolds}, volume~2 of {\em
  Atlantis Studies in Dynamical Systems}.
\newblock Atlantis Press, Paris, 2013.
\newblock The noncompact case.

\bibitem{Figalli-Shahgholian_2014}
A. Figalli and H. Shahgholian.
\newblock A general class of free boundary problems for fully nonlinear
  elliptic equations.
\newblock {\em Arch. Ration. Mech. Anal.}, 213(1):269--286, 2014.

\bibitem{GilbargTrudinger}
D.~Gilbarg and N.~S. Trudinger.
\newblock {\em Elliptic partial differential equations of second order}.
\newblock Classics in Mathematics. Springer-Verlag, Berlin, 2001.
\newblock Reprint of the 1998 edition.

\bibitem{Iliin-Shismarev1961}
V.~A. Il'in and I.~A. \v{S}i\v{s}marev.
\newblock The method of potentials for the problems of {D}irichlet and
  {N}eumann in the case of equations with discontinuous coefficients.
\newblock {\em Sibirsk. Mat. \v{Z}.}, pages 46--58, 1961.

\bibitem{Kim-Lee-Shahgholian2019}
S. Kim, K.A. Lee, and H. Shahgholian.
\newblock Nodal sets for ``broken'' quasilinear {PDE}s.
\newblock {\em Indiana Univ. Math. J.}, 68(4):1113--1148, 2019.

\bibitem{Lee1998}
K.A. Lee.
\newblock {\em Obstacle problems for the fully nonlinear elliptic operators}.
\newblock ProQuest LLC, Ann Arbor, MI, 1998.
\newblock Thesis (Ph.D.)--New York University.

\bibitem{Lee-Shahgholian2001}
K.A. Lee and H. Shahgholian.
\newblock Regularity of a free boundary for viscosity solutions of nonlinear
  elliptic equations.
\newblock {\em Comm. Pure Appl. Math.}, 54(1):43--56, 2001.

\bibitem{Li-Vogelius2000}
Y. Li and M. Vogelius.
\newblock Gradient estimates for solutions to divergence form elliptic
  equations with discontinuous coefficients.
\newblock {\em Arch. Ration. Mech. Anal.}, 153(2):91--151, 2000.

\bibitem{Li-Nirenberg2003}
Y. Li and L. Nirenberg.
\newblock Estimates for elliptic systems from composite material.
\newblock volume~56, pages 892--925. 2003.
\newblock Dedicated to the memory of J\"{u}rgen K. Moser.

\bibitem{Lin1986}
F. H. Lin
\newblock Second derivative $L^p$-estimate for elliptic equations of nondivergent type.
\newblock {\em Proceedings of the Amer. Math. Soc..}, 996:447--451, 1986.

\bibitem{Lions1955}
J.~L. Lions and L.~Schwartz.
\newblock Probl\`emes aux limites sur des espaces fibr\'{e}s.
\newblock {\em Acta Math.}, 94:155--159, 1955.

\bibitem{Lions83}
P.-L. Lions.
\newblock A remark on {B}ony maximum principle.
\newblock {\em Proc. Amer. Math. Soc.}, 88(3):503--508, 1983.

\bibitem{Nicolau-Soler2019}
A. Nicolau and O. Soler~i Gibert.
\newblock Approximation in the zygmund class.
\newblock {\em Journal of the London Mathematical Society}, 101(1):226?246, Jul
  2019.

\bibitem{Oleinik1961}
O.~A. Ole\u{\i}nik.
\newblock Boundary-value problems for linear equations of elliptic parabolic
  type with discontinuous coefficients.
\newblock {\em Izv. Akad. Nauk SSSR Ser. Mat.}, 25:3--20, 1961.

\bibitem{Petrosyan-Shahgholian-Uraltseva2012}
A. Petrosyan, H. Shahgholian, and N. Uraltseva.
\newblock {\em Regularity of free boundaries in obstacle-type problems}, volume
  136 of {\em Graduate Studies in Mathematics}.
\newblock American Mathematical Society, Providence, RI, 2012.

\bibitem{Picone1954}
M~Picone.
\newblock Sur un probl{\`e}me nouveau pour l'{\'e}quation lin{\'e}aire aux
  d{\'e}riv{\'e}es partielles de la th{\'e}orie math{\'e}matique classique de
  l'{\'e}lasticit{\'e}.
\newblock In {\em Colloque sur les {\'e}quations aux d{\'e}riv{\'e}es
  partielles, CBRM, Bruxelles}, pages 9--11, 1954.

\bibitem{Picone1936}
M. Picone.
\newblock Nuovi indirizzi di ricerca nella teoria e nel calcolo delle soluzioni
  di talune equazioni lineari alle derivate parziali della fisica-matematica.
\newblock {\em Ann. Scuola Norm. Super. Pisa Cl. Sci. (2)}, 5(3-4):213--288,
  1936.

\bibitem{Pimentel-Santos2018}
E.~A. Pimentel and M.~S. Santos.
\newblock Asymptotic methods in regularity theory for nonlinear elliptic
  equations: a survey.
\newblock In {\em P{DE} models for multi-agent phenomena}, volume~28 of {\em
  Springer INdAM Ser.}, pages 167--194. Springer, Cham, 2018.

\bibitem{Savin-Yu2019}
O. Savin and H. Yu.
\newblock Regularity of the singular set in the fully nonlinear obstacle
  problem, 2019.

\bibitem{Schechter1960}
M. Schechter.
\newblock A generalization of the problem of transmission.
\newblock {\em Ann. Scuola Norm. Sup. Pisa Cl. Sci. (3)}, 14:207--236, 1960.

\bibitem{Sheftel1963}
Z.~G. \v{S}eftel'.
\newblock Estimates in {$L\sb{p}$} of solutions of elliptic equations with
  discontinuous coefficients and satisfying general boundary conditions and
  conjugacy conditions.
\newblock {\em Soviet Math. Dokl.}, 4:321--324, 1963.

\bibitem{Silvestre-Sirakov2014}
L. Silvestre and B. Sirakov.
\newblock Boundary regularity for viscosity solutions of fully nonlinear
  elliptic equations.
\newblock {\em Comm. Partial Differential Equations}, 39(9):1694--1717, 2014.

\bibitem{Silvestre-Teixeira_2015}
L. Silvestre and E. Teixeira.
\newblock Regularity estimates for fully non linear elliptic equations which
  are asymptotically convex.
\newblock In {\em Contributions to nonlinear elliptic equations and systems},
  volume~86 of {\em Progr. Nonlinear Differential Equations Appl.}, pages
  425--438. Birkh\"{a}user/Springer, Cham, 2015.

\bibitem{Stampacchia1956}
G. Stampacchia.
\newblock Su un problema relativo alle equazioni di tipo ellittico del secondo
  ordine.
\newblock {\em Ricerche Mat.}, 5:3--24, 1956.

\bibitem{swiech97}
A. \'Swi\polhk{e}ch.
\newblock {$W^{1,p}$}-interior estimates for solutions of fully nonlinear, uniformly elliptic equations.
\newblock {\em Adv. Differential Equations}, 2(6):1005--1027, 1997.

\bibitem{Teixeira2016}
E.~Teixeira.
\newblock Geometric regularity estimates for elliptic equations.
\newblock In {\em Mathematical {C}ongress of the {A}mericas}, volume 656 of
  {\em Contemp. Math.}, pages 185--201. Amer. Math. Soc., Providence, RI, 2016.

\bibitem{Teixeira-Urbano_TA}
E.~Teixeira and J.M. Urbano.
\newblock Geometric tangential analysis and sharp regularity for degenerate
  {P}{D}{E}s.
\newblock In {\em Proceedings of the {I}{N}d{A}{M} {M}eeting ``{H}arnack
  {I}nequalities and {N}onlinear {O}perators" in honour of {P}rof. {E}.
  {D}i{B}enedetto}, Springer INdAM Ser. Springer, Cham, To appear.

\bibitem{Zygmund2002}
A.~Zygmund.
\newblock {\em Trigonometric series. {V}ol. {I}, {II}}.
\newblock Cambridge Mathematical Library. Cambridge University Press,
  Cambridge, third edition, 2002.
\newblock With a foreword by Robert A. Fefferman.

\end{thebibliography}
\end{document}